\documentclass[11pt,a4paper]{article}
\usepackage{amssymb,amsmath,amsthm,mathtools,a4wide,xypic,color}
 \usepackage{tikz,tikz-cd} \usetikzlibrary{babel}
\renewcommand{\baselinestretch}{1.15}\parskip=3pt
\mathtoolsset{showonlyrefs}
\usepackage{hyperref}
\hypersetup{
 colorlinks,
 citecolor=green,
 linkcolor=blue,
 urlcolor=Blue}
\renewenvironment{thebibliography}[1]{
  \begin{oldthebibliography}{#1}
    \setlength{\itemsep}{0em}
    \setlength{\parskip}{.3em}
}
{
  \end{oldthebibliography}
}
\makeatletter
\renewcommand{\paragraph}{%
  \@startsection{paragraph}{4}%
  {\z@}{1ex \@plus .5ex \@minus .2ex}{-1em}%
  {\normalfont\normalsize\bfseries}%
}
\makeatother

\DeclareSymbolFont{boldoperators}{OT1}{cmr}{bx}{n}
\SetSymbolFont{boldoperators}{bold}{OT1}{cmr}{bx}{n}
\edef\boldbar{\unexpanded{\protect\mathaccentV{bar}}\number\symboldoperators16}

\newtheorem{thm}{Theorem}
\newtheorem*{teo*}{Theorem}
\newtheorem{corol}[thm]{Corollary}
\newtheorem{lemma}[thm]{Lemma} 
\newtheorem{prop}[thm]{Proposition}
\newtheorem{defin}[thm]{Definition}
\theoremstyle{definition}

\theoremstyle{remark}
\newtheorem{remqee}[thm]{Remark}

\newenvironment{rem}{\begin{remqee}}{\qee\end{remqee}}
\numberwithin{equation}{section}

\newcommand{\A}{\mathcal A}
\newcommand{\W}{\mathcal W}
\newcommand{\X}{\mathcal X}
\newcommand{\U}{\mathcal U}
\newcommand{\V}{\mathcal V}
\renewcommand{\P}{\mathbb P}
\newcommand{\R}{\mathbb R}
\newcommand{\C}{\mathbb C}
\newcommand{\Z}{\mathbb Z}
\newcommand{\K}{\mathfrak K}
\newcommand{\rk}{\operatorname{rk}}
\newcommand{\qee}{\ \hfill $\triangle$}
\newcommand{\Kom}{{\mathbf{Kom}}}
\newcommand{\Hom}{\operatorname{Hom}}
\newcommand{\id}{\operatorname{id}}
\newcommand{\bnu}{{\boldsymbol\nu}}
\newcommand{\fr}{\heartsuit}
\newcommand{\cO}{{\mathcal O}}
\newcommand{\cE}{{\mathcal E}}
\newcommand{\opq}[2]{\cO_{\Sigma_n}(#1,#2)}
\newcommand{\GL}{\operatorname{GL}}
\newcommand{\Hilb}{\operatorname{Hilb}}

\begin{document}

 \begin{center} \bf     NESTED HILBERT SCHEMES  ON HIRZEBRUCH \\[5pt] SURFACES  AND QUIVER VARIETIES \end{center}

\thispagestyle{empty} \vspace{-3mm}
\begin{center}{\sc Ugo Bruzzo,$^{abcd}$    Valeriano Lanza,$^{e}$ and   {\sc Pedro Henrique Dos Santos}$^{b}$}   \\[5pt]
\small
$^a$ SISSA (International School for Advanced Studies), Via Bonomea 265, 34136 Trieste, Italy \\
$^b$ Departamento de Matem\'atica, Universidade Federal da Para\'iba, \\ Campus I, Castelo Branco, Jo\~ao Pessoa, PB, Brazil \\
$^c$ INFN (Istituto Nazionale di Fisica Nucleare), Sezione di Trieste \\
$^d$ IGAP (Institute for Geometry and Physics), Trieste  \\
$^e$ Instituto de Matem\'atica e Estat\'istica, Universidade Federal \\ Fluminense,  Rua Prof.~M.F.~de Freitas Reis, Niter\'oi, RJ, Brazil \\[5pt]

Email: {\tt bruzzo@sissa.it, vlanza@id.uff.br, pedro.santos@academico.ufpb.br}
\end{center}

\vfill

\begin{abstract}   For $n\ge 1$ we show that the length 1 nested Hilbert scheme of the total space $\Xi_n$ of the line bundle $\cO_{\P^1}(-n)$, parameterizing pairs of nested 0-cycles in $\Xi_n$, is a quiver variety associated with a suitable quiver with relations. This generalizes previous work about nested Hilbert schemes on $\C^2$ in one direction, and about the Hilbert schemes of points of $\Xi_n$ in another direction.
\end{abstract}

\vfill
%
\renewcommand{\baselinestretch}{0.25}
\def\contentsname{\normalsize Contents}
{\footnotesize\tableofcontents}
\renewcommand{\baselinestretch}{1.15}

\vfill\noindent\begin{minipage}{\textwidth} \small
\parbox{\textwidth}{\hrulefill} \\
Date:  March 26th, 2024. Revised May 7th, 2024
U.B.'s research is partly supported by   PRIN 2022 ``Geometry of algebraic structures: moduli, invariants, deformations'' and INdAM-GNSAGA.
Moreover, this investigation started while he was the recipient of the  CNPq ``Bolsa de Produtividade em Pesquisa  1B'' 313333/2020.  V.L.'s research is partly supported by the FAPERJ research grant ``Jovem Cientista de Nosso Estado''  n.~E-26/200.280/2023.
 P.H.S. was supported by a CAPES PhD Fellowship (2020-2024)  and has benefitted of the CNPq SWE fellowship n.~200861/2022-0.
\end{minipage}

\newpage

\section{Introduction}

Quivers are magic. Starting from an often very simple directed graph one first constructs an associative algebra, then constructs a moduli space of its representations (called a {\em quiver variety}), and eventually discovers that this is also the moduli space of nontrivial geometric structures. For instance, if we
consider instantons, i.e. anti-self-dual connections on the 4-sphere, including degenerate configurations, we have a singular moduli space, whose resolution of singularities on the one hand, is the moduli space of framed torsion-free sheaves on the complex projective plane; and on the  other hand, thanks to Nakajima's work  \cite{Nakabook}, can be regarded as the moduli space of representations of the path algebra of a quiver with relations --- the ADHM quiver. Since a rank one torsion-free sheaf on $\P^2$, framed on a line, may be identified with the ideal sheaf of a 0-cycle on $\C^2$, also the Hilbert scheme of points of $\C^2$ is a quiver variety. 
Other examples of this correspondence are\vskip-25pt
t\begin{itemize} \itemsep=-2pt 
\item moduli spaces of instantons on ALE spaces \cite{Kronh-Naka};
\item (equivariant) Hilbert schemes of points of ALE spaces (\cite{Kuz} and references therein);
\item the crepant resolutions of singularities $\C^3/G$, where $G$ is a finite subgroup of $SL_3(\C)$, are moduli spaces of structures called $G$-constellations (a generalization of the $G$-Hilbert schemes), and are moduli spaces of representations of the McKay quivers \cite{Craw-Ishii, CMT}. Related constructions and results can also be found in \cite{CGGS1,CGGS2} and other papers.
\end{itemize} 

Another construction is described in \cite{JvF,JvFL}. One considers {\em framed flags} on $\P^2$, i.e., pairs $(E,F)$, where
$E$ and $F$ are torsion-free sheaves of the same rank on $\P^2$, such that $E\subset F$, $F$ is framed on a line, and the quotient $F/E$ has dimension zero and is supported away from the line. The moduli space of these pairs turns out to be   a quiver variety associated with an {\em enhanced ADHM quiver.} When $\rk E = \rk F = 1$ the moduli space is the {\em nested Hilbert scheme} of $\C^2$, and parameterizes pairs of nested 0-cycles. (Nested Hilbert schemes were probably first considered by Keel \cite{Keel} and were studied in some detail  by Cheah in \cite{Cheah-Nested1}.)

By removing a suitable  rational curve $\ell_\infty$ from the $n$-th Hirzebruch surface $\Sigma_n$, with $n\ge 1$, one obtains the total space of the line bundle $\cO_{\P^1}(-n)$, that we denote $\Xi_n$. Rank 1 torsion-free sheaves on $\Sigma_n$ framed on $\ell_\infty$, suitably twisted, are ideal sheaves of 0-cycles in $\Xi_n$.  Building on the monadic description of the moduli spaces of framed sheaves on $\Sigma_n$ that was performed in
 \cite{bbr}, in \cite{bblr,bblr-irr} it was shown that the Hilbert schemes of points $\operatorname{Hilb}^c(\Xi_n)$ are quiver varieties associated with a suitable quiver ${\mathbf Q}^n$, which we call the {\em $n$-th Hirzerbruch quiver.} The aim of the present paper is to obtain a similar description in the case of the {\em nested Hilbert schemes of points} $ \operatorname{Hilb}^{c'\!,c}(\Xi_n)$.
 
 Since we are going to use a categorical language (so the main problem is to establish an isomorphism between the functor of families of  representations of a certain quiver, and the functor of families of nested 0-cycles on $\Xi_n$), we start by rephrasing the results of both \cite{JvF} and \cite{bbr} in full categorical language, in Sections \ref{projplane} and \ref{framHirz}, respectively.  This is preceded by Section \ref{gen} where a few generalities about spaces of representations of a quiver and about framed flags of sheaves are recalled. In Section  \ref{trippa}   the main result is proved. The trick for doing that is the same as in the case of the projective plane, i.e., to regard the spaces of representations of an ``enhanced'' quiver as a space of morphisms between two copies of the quiver ${\mathbf Q}^n$, although the present case is more complicated and technically more involved.

\smallskip
 \noindent{\bf Acknowledgements.} We thank Simone Marchesi for a useful conversation. The research on which this paper is based was initiated while U.B.~was a Professor Titular Visitante at Universidade Federal da Para\'iba, Jo\~ao Pessoa, Brazil; he thanks the colleagues at the  Department of Mathematics for support and for their hospitality. P.H.S.~thanks SISSA for the warm hospitality while he was visiting under a CNPq grant
``Programa institucional de doutorado sandu\'iche no exterior.''
 \section{Generalities } \label{gen}
We start this  section  by reminding some generalities about families of quiver representations, basically following \cite{KiQ}. Then we introduce the functor
of families of framed flags on a projective surface.  By {\em scheme} we shall always mean a connected   scheme of finite type over $\C$. All locally free sheaves will have finite rank.

\subsection{Quiver representations} Let ${\mathbf Q}$ be a quiver; ${\mathbf Q}_0$ will denote the set of vertexes, and ${\mathbf Q}_1$ the set of arrows. A stability parameter $\Theta$ for ${\mathbf Q}$ may be regarded as an element in $\R^{\# {\mathbf Q}_0}$. The maps $s,t\colon \mathbf Q_1\to \mathbf Q_0$
are the {\em source} (tail) and {\em target} (head) maps, respectively. 
\begin{defin} 1. A family of representations of ${\mathbf Q}$ parameterized by a scheme $T$ is, for every $v\in {\mathbf Q}_0$,
a locally free sheaf $\mathcal W_v$ on $T$, and for any arrow   $a\in {\mathbf Q}_1$,  a sheaf morphism
$\phi_a \colon \W_{s(a)} \to\W_{t(a)}$. Note that for every closed point $t\in T$ by taking the fiber at $t$  one gets a representation of ${\mathbf Q}$ in the usual sense. 

2. A morphism between two families of representations $(T,\mathcal W_v,\phi_a)$ and $(S,\mathcal U_v,\psi_a)$ 
is a scheme morphism $f\colon T \to S$ and a collection of sheaf morphisms $\{F_v\colon\W_v \to f^\ast\U_v,\ v\in {\mathbf Q}_0\}$ such that for every arrow $a\in {\mathbf Q}_1$
the diagram
$$\xymatrix{\W_{s(a)} \ar[r]^{F_{s(a)\ }} \ar[d]_{\phi_a} & f^\ast \mathcal U_{s(a)} \ar[d]^{f^\ast\psi_a} \ar[d] \\
\mathcal W_{t(a)} \ar[r]^{F_{t(a)\ } } &   f^\ast \mathcal U_{t(a)}
}$$
commutes. A morphism is an isomorphism when $f$ and all morphisms $F_v$ are isomorphisms.

3. A family of representations is $\Theta$-stable if for every closed points $t\in $T the representation corresponding to $t$ is $\Theta$-stable. 
\end{defin}

Note that if $(S,\mathcal U_v,\psi_a)$ is a family of representations parameterized by $S$, and $f\colon T\to  S$ is a scheme morphism,
then $(T, f^\ast \mathcal U_v,f^\ast\psi_a)$ is  a family of representations parameterized by $T$.

\begin{defin} [The representation moduli functor]
The functor of   families of  representations of ${\mathbf Q}$ is the functor
\begin{equation}
\begin{array}{cccl}
\mathfrak R^{\mathbf Q} \colon & \mbox{\bf Sch}^{\mbox{\footnotesize\rm op}} & \to & \mbox{\bf Set} \\[10pt]
& T & \mapsto & \left\{\parbox{62mm}{\rm isomorphism classes of families of representations of ${\mathbf Q}$ parameterized by $T$}\right\}
\end{array}
\end{equation}
where $ \mbox{\bf Sch}$ is the category of connected   schemes of finite type over $\C$.
The action of this functor on morphisms is by pullbacks: if $f\in \operatorname{Hom}(T,S)$, then 
$$ \W_{s(a)} \xrightarrow{\phi_a} \W_{t(a)} \quad \text{is sent to} \quad   f^\ast\W_{s(a)} \xrightarrow{f^\ast \phi_a} f^\ast \W_{t(a)} .$$

\end{defin}
The {\em dimension vector}  $\mathbf v$ of a representation is the string of nonnegative integers
$(d_1,\dots, d_{\# {\mathbf Q}_0})$ where $d_i = \rk  \W_i$.   
After fixing a dimension vector $\mathbf v$ and a stability parameter $\Theta$, one can also introduce the subfunctor $\mathfrak R^{{\mathbf Q},s}_{\mathbf v, \Theta}$
of $\Theta$-stable $\mathbf v$-dimensional representations.
 If $\mathbf v$ is primitive, this functor is representable by a fine moduli space  
$\mathcal M^{s}_{\mathbf v, \Theta}$ \cite[Prop. 5.3]{KiQ}.

\paragraph{Framed representations.} Actually we shall be concerned with {\em framed representations} of quivers, in particular, representations that are framed at one vertex.  One chooses a vertex in the given quiver, and fixes the vector space $W$ associated with that vertex.  Note that that fixes one component of the dimension vector. Moreover, one only considers morphisms of representations such that the morphism corresponding to that vertex is either zero or the identity. For families, we fix the vector bundle corresponding to the framing vertex to be $W\otimes \cO_T$, where $W$ is a fixed vector space. 
The functor of families of framed representations will be denoted by $\mathfrak R^{{\mathbf Q}\fr }$.

\subsection{Framed flags}
We introduce now the notion of framed flag of sheaves (of length 1). Let $X$ be an irreducible projective smooth surface, and $D$ a divisor in it (for the moment we only establish some notation, and at this level of generality we do not need to make any additional assumptions on $X$ and $D$).
A framed flag of length 1 and type $(r,\gamma,c,\ell)$ on $(X,D)$
is a triple $(E,F,\phi)$, where 
\begin{itemize} \itemsep=-2pt
\item
$E$ and $F$ are torsion-free sheaves on $X$, 
with $E\subset F$, $r=\rk E =\rk F$;
\item the support of $F/E$ is 0-dimensional and is disjoint from $D$;
\item
$\phi$ is an isomorphism of $F_{\vert D}$ with $\cO_D^{\oplus r}$; \item
$c_1(F)=\gamma \in \operatorname{NS}(X)$;
$c = c_2(F)$, $\ell = c_2(E)-c_2(F) = h^0(X, F/E)$. 
\end{itemize} 
As a consequence, $\phi$ also provides an isomorphism $E_{\vert D} \simeq \cO_D^{\oplus r}$. Note that necessarily $c_1(E) = \gamma $ and  $\gamma\cdot D=0$. 

We define the functor   $\mathfrak F^{X,D}_{r,\gamma,c,\ell}\colon \mbox{\bf Sch}^{\mbox{\footnotesize\rm op}} \to \mbox{\bf Set}$ of families of   length 1 framed flags   on $(X,D)$
as
\begin{equation}  \mathfrak F^{X,D}_{r,\gamma,c,\ell}(T) = \{\mbox{isomorphism classes of triples}\ (E,F,\phi)\} \label{flagfunct} \end{equation}
where
\begin{itemize}\itemsep=-2pt \item $E$, $F$ are rank $r$ torsion-free sheaves on $X\times T$, flat on $T$,  with $E\subset F$;
\item  for all closed points $t\in T$, the support of $(F/E)_t$ is 0-dimensional and is disjoint from $D$;
\item $\phi$ is an isomorphism $\phi \colon F_{\vert D \times T} \to \mathcal O_{D\times T}^{\oplus r}$;
\item for all closed $t\in T$, $c_2(F_t) = c$, $c_2(E_t)-c_2(F_t) = \ell$, $c_1(F_t) = \gamma$; 
\item morphisms of families of framed flags are defined in the obvious way;
\item the functor acts on scheme morphisms by pullback.
\end{itemize}
This functor was  defined in \cite{JvF}
for $X=\P^2$ and $D$ a line, that we denote as usual $\ell_\infty$ (note that necessarily $\gamma=0$ in that case). Again in \cite{JvF}, it was proved that in that case this functor is representable. This may be generalized as follow.

\begin{thm} Let $X$ be a smooth, irreducible projective surface, and let $D$ be a smooth, irreducible, big and nef divisor in $X$. Then for every choice of $(r,\gamma,c,\ell)$, the functor  $\mathfrak F^{X,D}_{r,\gamma,c,\ell}$ is representable. \label{flagrepr}
\end{thm}

\begin{proof}According to Corollary 3.3 of \cite{BruMar}  there exists a fine moduli space of torsion-free sheaves $F$ on $X$, with invariants
$\rk F =r$, $c_1(F)=\gamma$, $c_2(F)=c$, framed on $D$ to the trivial sheaf. Then the proof of Proposition 1 in \cite{JvF} applies verbatim.
\end{proof}

\begin{rem} This theorem can be further generalized by replacing the trivial sheaf on $D$ with any semistable vector bundle of rank $r$. \end{rem}

We denote by $\mathcal F^{X,D}_{r,\gamma,c,\ell}$ the scheme   representing the functor $\mathfrak F^{X,D}_{r,\gamma,c,\ell}$. Not much is known about the smoothness, and, more generally, irreducibility and reducedness of this scheme when $r>1$ (see \cite{JvFL} in this connection). When $r=1$ one can assume $\gamma=0$, and then
$\mathcal F^{X,D}_{1,0,c,\ell}$ is the nested Hilbert scheme $\operatorname{Hilb}^{c,c+\ell}(X_0)$, parameterizing pairs of nested 0-cycles of length $c$ and $c+\ell$ in the quasi-projective smooth surface $X_0=X\setminus D$. One knows from \cite{Cheah-Nested1} that $\operatorname{Hilb}^{c,c+\ell}(X_0)$ is smooth if and only if $\ell=1$.

 \section{The case of the projective plane} \label{projplane}
 The enhanced ADHM quiver $\boldbar {\mathbf Q}$
 is the  quiver
\begin{equation} \label{augADHMquiv}
	\begin{tikzcd}
		\bullet \arrow[in=240, out=300,loop, "b_2'"]\arrow[out=120,in=60,loop,"b_1'"] \arrow[rr, "\phi"]  && \bullet  \arrow[out=300,in=240,loop, "b_1"]\arrow[out=120,in=60,loop,"b_2"] \arrow[out=30, rr, in=150,"i"]  &	 &\bullet  \arrow[out=210,ll, in=330,"j"] \ \infty	\end{tikzcd}
\end{equation}
 with the relations 
 \begin{equation} \label{relaugADHM}b_1b_2-b_2b_1+ij=0; \quad b_1\phi-\phi b_1'=0; \quad  b_2\phi-\phi b_2'=0;\quad j\phi=0;\quad b_1'b_2'-b_2'b_1'=0.\end{equation}
 Jardim and von Flach in \cite{JvF} proved that in the case $(X,D) = (\P^2,\ell_\infty)$, for $\mathbf v = (\ell,c+\ell,r)$, and with a suitable choice of the stability parameter $\Theta$, the functor of families of stable framed representations
 $\mathfrak R^{\boldbar {\mathbf Q}\fr  s}_{\mathbf v, \Theta} $ and the functor of families of framed flags $\mathfrak F_{r,c,\ell}=\mathfrak F^{\P^2,\ell_\infty}_{r,0,c,\ell}$ are isomorphic.
 We review here their proof, providing some more details, especially about the categorical formalization of the problem. This will be useful, as a comparison but also to provide some needed results, for what we shall do in the next sections about framed flags on Hirzebruch surfaces. 
 
The first step will be to represent $\mathfrak F_{r,c,\ell}$ as a functor of families of representations of $\boldbar {\mathbf Q}$.
 The components of the dimension vector of this quiver list the dimensions of the vector spaces attached to the vertexes from left to right.

The crux of the above mentioned result is the following theorem. 
\begin{thm} {\rm \cite{JvF,JvFL} }\label{JvFthm}
 Let $\mathbf v = (\ell,c+\ell,r)$, and 
let $\Theta = (\theta,\theta',\theta_\infty)\in\R^3$ 
 with $$\theta'>0, \qquad \theta+\theta'<0 \quad\text{and}\quad \ell\theta+(c+\ell)\theta'+r\theta_\infty=0.$$
 Let $ \mathfrak R^{\boldbar {\mathbf Q}\fr  s}_{\mathbf v, \Theta}$ be the functor of families of  framed representations
 of the   enhanced ADHM quiver   $\boldbar {\mathbf Q}$ depicted in equation \eqref{augADHMquiv} with the relations \eqref{relaugADHM}, framed at the vertex $\infty$.\footnote{Note that the vector space $W$ corresponding to the framing vertex has dimension $r$.}
  There exists a  natural transformation $\eta\colon \mathfrak R^{\boldbar {\mathbf Q}\fr  s}_{\mathbf v, \Theta}\to \mathfrak F_{r,c,\ell}$ 
which is an isomorphism of functors.
\label{isofunct}\end{thm}

We shall also need to consider the standard ADHM quiver, which we shall denote by ${\mathbf Q}$:
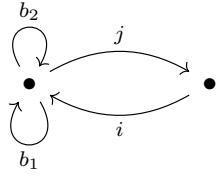
\begin{equation} \label{ADHMquiv}
	\begin{tikzcd}
				\bullet  \arrow[out=300,in=240,loop, "b_1"]\arrow[out=120,in=60,loop,"b_2"] \arrow[out=30, rr, in=150,"j"]  &	 &\bullet \arrow[out=210,ll, in=330,"i"]
	\end{tikzcd}
\end{equation}
with the relation 
\begin{equation}   b_1b_2-b_2b_1+ij=0. \label{smallrels} \end{equation}

We develop now some theory which will be needed to prove Theorem \ref{JvFthm}. 
We introduce the following categories:
\begin{itemize}\itemsep=-2pt \item 
the category $\A _{\mathbf Q}$ of families of representations of the ADHM quiver ${\mathbf Q}$ with the relations \eqref{smallrels}. An object  in $\A _{\mathbf Q}$ is a collection  $(T,\V,\W,B_1,B_2,I,J)$, where $T$ is a scheme, 
 $\V$ and $\W$ are vector bundles on $T$, and 
 $$
 B_1,B_2 \in \operatorname{End}(\V), \quad I \in \operatorname{Hom}(\W,\V),\quad  J \in \operatorname{Hom}(\V,\W)$$
 satisfying the condition
 \begin{equation}\label{relADHM}  B_1B_2-B_2B_1 +IJ =0. \end{equation}
{Let $\A _{\mathbf Q}^s$ be the full subcategory of families of representations
 that are stable with respect to the standard stability condition \cite{Nakabook}.}
 \item The category $\Kom_{\P^2}$  of   families of complexes of coherent sheaves on $\mathbb P^2$. Objects are given by a scheme $T$ and a complex of coherent sheaves on $T\times \mathbb P^2$; the morphisms are the obvious ones. $\Kom_{\P^2}^{\mbox{\tiny flat}}$ is the full subcategory of   families of complexes whose cohomology sheaves are flat over $T$. 
\end{itemize} 
$\A _{\mathbf Q}$ and $\Kom_{\P^2}$  are categories   over the category $\mbox{\bf Sch}$ of schemes.\footnote{Actually, since they admit pullbacks, both categories are fibered categories over  $\mbox{\bf Sch}$.
See \cite{Vistoli-notes}, Definition 3.5, or \cite{stacks-project}, Section 4.33.}
 Their fiber categories over $T=\operatorname{Spec} \C$ are the category
of representations of the ADHM quiver ${\mathbf Q}$   (and then $\V$, $\W$ are just vector spaces) and the category of complexes of coherent sheaves over $\P^2$, respectively. If $T$ is a scheme, we denote by $\A_{\mathbf Q}(T)$ the fiber of $\A_{\mathbf Q}$ over $T$, i.e., the category of families of representations of ${\mathbf Q}$ parameterized by $T$, with a similar meaning for $\Kom_{\P^2}(T)$.

\begin{rem}\label{remNaka}
By  Nakajima's work we know that, fixing the dimension vector $\mathbf v=(c,r)$,  the corresponding functor of families of  stable representations of the quiver ${\mathbf Q}$ is represented by a scheme which
 is isomorphic to the moduli space $\mathcal M(r,c)$ of  torsion-free sheaves on $\P^2$, of rank $r$ and second Chern class $c$, with a framing to the trivial sheaf on a fixed line.
 \end{rem}

We   introduce a functor $$\K_{\mathbf Q} \colon \A _{\mathbf Q} \to\Kom_{\P^2}$$ of categories over 
$\mathbf{Sch}$; this is a relative version of the ``absolute'' standard functor 
 which  associates a complex with a   representation of the ADHM quiver.  The functor $\K_{\mathbf Q}$
 associates with a family of   representations of ${\mathbf Q}$ parameterized by a scheme $T$ the corresponding family of 3-term complexes on $\P^2\times T$. Note that as we are not requiring   the representations to be  stable the 3-term complex may have nontrivial cohomology in every degree, in particular, it may not be a monad. If $\X=(T,\V,\W,B_1,B_2,I,J)$ is an object in $\A _{\mathbf Q}$, 
   then  $\K_{\mathbf Q}(\X)$ is the following complex supported in degree $-1$, 0 and 1, whose terms are sheaves on $T\times\P^2$:
 $$ 0 \to  \V \boxtimes  \cO_{\P^2}(-1) \xrightarrow{\ \alpha\ }  (\V\oplus \V \oplus \W)\boxtimes \cO_{\P^2}\xrightarrow{\ \beta\ }  \V \boxtimes   \cO_{\P^2}(1)\to 0$$
where  the morphisms $\alpha$, $\beta$ are given by
 $$ \alpha = \begin{pmatrix} zB_1 +x 1_{\V} \\ zB_2 + y1_{\V} \\ zJ \end{pmatrix}, \qquad
 \beta = \begin{pmatrix} -zB_2 -y1_{\V} ,\ zB_1+x1_{\V},\ zI\end{pmatrix}  $$
 with $(x,y,z)$ homogeneous coordinates in $\P^2$.  Note that $\beta\circ\alpha=0$ due to the relation \eqref{relADHM}.
 
{ A morphism $\xi= (f,\xi_1,\xi_2)$  of families of  representations
 $$ \X=(S,\V,\W,B_1,B_2,I,J) \xrightarrow{\ \xi \ }  \tilde{ \X}=(T,\tilde \V,\tilde \W,\tilde B_1,\tilde B_2,\tilde I, \tilde J)$$
 is a morphism $f\colon S\to T $ and a pair of morphisms $\xi_1 \colon \V \to f^\ast\tilde \V$, $\xi_2 \colon \W \to f^\ast\tilde \W$
 satisfying
 $$\xi_1\circ B_1 = f^\ast\tilde B_1 \circ \xi_1,\qquad \xi_1\circ B_2 = f^\ast\tilde B_2 \circ \xi_1, \qquad \xi_2\circ J = f^\ast\tilde J \circ \xi_1, \qquad \xi_1\circ I =f^\ast \tilde I \circ \xi_2.$$}
 The morphism $\K_{\mathbf Q}(\xi)\colon \K_{\mathbf Q}(\X) \to \K_{\mathbf Q}(\tilde{\X})$ between the corresponding monads is given by the diagram
\begin{equation} \label{morP2}\xymatrix{
0 \ar[r] &  \V \boxtimes  \cO_{\P^2}(-1) \ar[r]^(.42)\alpha \ar[d]_{\xi_1 \times \operatorname{Id}} &
  (\V\oplus \V \oplus \W) \boxtimes \cO_{\P^2} \ar[r]^(.58)\beta \ar[d]_{(\xi_1 \oplus \xi_1 \oplus \xi_2)\times  \operatorname{Id} }&
  \V \boxtimes \cO_{\P^2}(1) \ar[d]^{\xi_1 \times \operatorname{Id}} \ar[r] & 0\\
0 \ar[r] &  f^\ast\tilde \V \boxtimes   \cO_{\P^2}(-1) \ar[r]^(.42){f^\ast\tilde\alpha} & 
 f^\ast  (\tilde\V\oplus \tilde\V \oplus \tilde\W)\boxtimes \cO_{ \P^2}  \ar[r]^(.58){f^\ast\tilde\beta} &  
   f^\ast\tilde \V \boxtimes  \cO_{\P^2}(1)  \ar[r] & 0
    }
  \end{equation}
 
 \medskip
 \begin{prop} For every scheme $T$, the functor $\K_{\mathbf Q} (T)\colon \A _{\mathbf Q}(T) \to\Kom_{\P^2}(T)$ is exact and   faithful.
 \end{prop}
 \begin{proof} The proof  of  Proposition 2.3.5 in \cite{Pat} applies verbatim. 
 \end{proof}
 
The next result requires that the representations we consider are framed and stable. So we define $\A_{\mathbf Q}^{\fr s}$ as the subcategory
of $\A_{\mathbf Q}$ whose objects are families of  framed representations of ${\mathbf Q}$, stable with respect to the standard stability condition. Note that this  is not  a full subcategory as the morphisms at the framing vertex are restricted.
 
 \begin{prop}  $\K_{\mathbf Q}$ maps the subcategory $\A_{\mathbf Q}^{\fr s}$ to the subcategory $\Kom_{\P^2}^{\mbox{\tiny \rm flat}}$. \label{propflat}
 \end{prop}
 
 \begin{proof} The stability of the family of representations on which we act by $\K_{\mathbf Q}$ implies that the morphism $\alpha$ is injective and $\beta$ is surjective.
 Then we may reduce to prove  the following fact: if
 $$ 0 \to \cE ' \xrightarrow{\ \alpha\ } \cE   \xrightarrow{\ \beta\ } \cE '' \to 0 $$
 is a complex of families of locally free coherent sheaves on $T\times\P^2$, with $\alpha$   injective and $\beta$   surjective,   then the cohomology sheaf
 $\mathcal H = \ker \beta/\operatorname{im} \alpha$ is flat over $T$.  To prove this we first consider the exact sequence
 $$ 0 \to  \ker \beta \to  \cE  \to \cE '' \to 0 ,$$
 where  $ \cE  $ and $ \cE ''$ are flat over $T$, so that $\ker \beta$ is flat as well. Then one applies Lemma 2.1.4 in \cite{HL-book}
 to the exact sequence
 $$ 0 \to \operatorname{im} \alpha \to \ker\beta \to \mathcal H \to 0 .$$
\end{proof} 

\begin{rem} The image $\K_{\mathbf Q}(\A _{\mathbf Q}^{\heartsuit s})$ is the subcategory of $\Kom_{\P^2}^{\mbox{\tiny \rm flat}}$
whose objects are families of monads for the ADHM quiver (in particular their cohomology is flat over $T$).
\end{rem} 
 
 Now we construct the natural transformation $\eta\colon \mathfrak R^{\boldbar {\mathbf Q}\fr  s}_{\mathbf v, \Theta}\to \mathfrak F_{r,c,\ell}$. 
 The trick for doing that is to regard a representation of the enhanced ADHM quiver as a morphism of representations of the standard  ADHM quiver. Let
 $   (T,\V',\V,\W,B_1',B_2',B_1,B_2,I,J,\Phi)$ be a family of   representations of the enhanced ADHM quiver, framed at the vertex 0. So $T$ is a scheme, and $\V'$ and $\V$ are vector bundles on $T$ of rank $c$ and $c+\ell$, respectively. $\W$ is the trivial bundle $W\otimes \cO_T$ for
 some fixed vector space $W$ of dimension $r$. Moreover,
 $$B_1',B_2'\in\operatorname{End}(\V'),\quad B_1,B_2\in\operatorname{End}(\V),\quad I\in \operatorname{Hom}(\W,\V),\quad J\in \operatorname{Hom}(\V,\W),\quad\Phi\in \operatorname{Hom}(\V',\V).$$
 Assume that this representation is stable as in Theorem \ref{JvFthm}. This implies that $\Phi$ is injective, and then $\Phi$  defines a morphism of  families of representations of the standard ADHM quiver described by the following diagram
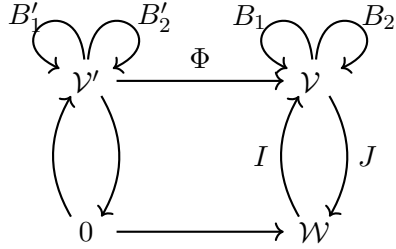
\begin{equation}
	\begin{tikzpicture} 
		
		\node at (0,0) {$\V'$} edge [thick, out=95, in=155, loop] () edge [thick, out=85, in=25, loop]  () ;  \node at (0,-2) {$0$}; \node at (-0.8, 0.85) {$B_1'$}; \node at (0.9, 0.85){$B_2'$};
		\draw[->,thick,bend left] (0.2,-0.2) to (0.2,-1.8);
		\draw[->,thick,bend left] (-0.2,-1.8) to (-0.2,-0.2); 
		\node at (3,0) {$\V\phantom{'}$} edge [thick, out=95, in=155, loop] () edge [thick, out=85, in=25, loop]  (); \node at (2.15, 0.85) {$B_1$}; \node at (3.85, 0.85){$B_2$}; \node at (3,-2) {$\W$};    
		\draw[->,thick,bend left] (3.2,-0.2) to (3.2,-1.8);
		\draw[->,thick,bend left] (2.8,-1.8) to (2.8,-0.2);
		\draw[->,thick] (0.4,0) to (2.6,0); \draw[->,thick] (0.4,-2) to (2.6,-2);
		\node at (1.5,0.3) {$\Phi$} ; \node at (2.3,-1) {$I$};  \node at (3.7,-1) {$J$};
	\end{tikzpicture}
\end{equation}
Let $\V'' = \V/\Phi(\V')$; {note that $\V''$ is locally free (of rank $\ell$) as $\Phi$ is injective on every fiber of $\V'$. } The morphisms $B_1$, $B_2$, $B_1'$, $B_2'$, $I$, $J$ induce morphisms
$$B_1'',B_2''\in\operatorname{End}(\V''),\quad I''\in \operatorname{Hom}(\W,\V''),\quad J''\in \operatorname{Hom}(\V'',\W) $$
which define a quotient family of representations of the ADHM quiver. This is represented by the diagram
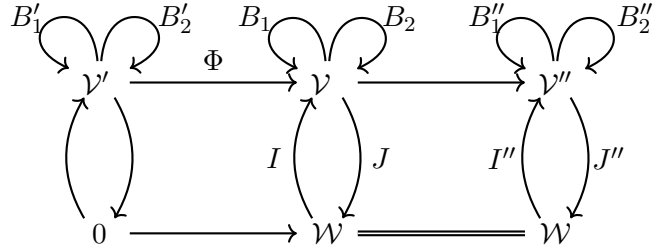
\begin{equation}
	\begin{tikzpicture} 
		\node at (0,0) {$\V'\phantom{'}$}  edge [thick, out=95, in=155, loop] () edge [thick, out=85, in=25, loop]  () ; \node at (0,-2) {$0$};  \node at (-0.95, 0.85) {$B_1'$}; \node at (1, 0.85){$B_2'$};  
		\draw[->,thick,bend left] (0.2,-0.2) to (0.2,-1.8);
		\draw[->,thick,bend left] (-0.2,-1.8) to (-0.2,-0.2); 
		\node at (3,0) {$\V\phantom{''}$} edge [thick, out=95, in=155,loop] () edge [thick, out=85, in=25, loop] (); \node at (3,-2) {$\W$};  \node at (2.05, 0.85) {$B_1$}; \node at (3.95, 0.85){$B_2$};  
		\draw[->,thick,bend left] (3.2,-0.2) to (3.2,-1.8);
		\draw[->,thick,bend left] (2.8,-1.8) to (2.8,-0.2);
		\draw[->,thick] (0.4,0) to (2.6,0); \draw[->,thick] (0.4,-2) to (2.6,-2);
		\node at (1.5,0.3) {$\Phi$} ; \node at (2.3,-1) {$I$};  \node at (3.7,-1) {$J$};
		\node at (6,0) {$\V''$} edge [thick, out=95, in=155, loop] () edge [thick, out=85, in=25, loop]  (); \node at (6,-2) {$\W$};  \node at (5.1, 0.85) {$B_1''$}; \node at (7.05, 0.85){$B_2''$};  
		\draw[->,thick,bend left] (6.2,-0.2) to (6.2,-1.8);
		\draw[->,thick,bend left] (5.8,-1.8) to (5.8,-0.2);
		\draw[->,thick] (3.4,0) to (5.6,0); \draw[style=double,thick] (3.4,-2) to (5.6,-2);
		\node at (6.7,-1) {$J''$};  \node at (5.3,-1) {$I''$};
	\end{tikzpicture}
\end{equation}
i.e., we have an exact sequence of families of representations of the standard ADHM quiver
$$ 0 \to \X' \to \X \to\X'' \to 0.$$
Here $ \X$ and  $\X''$ are families of stable representations (for the stability of $\X''$ see \cite[p.~148]{JvF}). 
Applying the exact functor $\K_{\mathbf Q}$ we obtain an exact sequence of complexes of coherent sheaves  on $T\times\P^2$
$$ 0 \to E_{\X'} \to E_{\X} \to E_{\X''} \to 0 $$
whose nonzero terms are in degree $-1$, 0 and 1.  This exact sequence of complexes makes up the following commutative diagram with exact rows, whose columns are the complexes corresponding to  $\X' $, $\X$, $\X'' $, respectively:
$$\xymatrix{& 0 \ar[d]  & 0 \ar[d] & 0 \ar[d] \\ 
0 \ar[r] & \V '\boxtimes \cO_{\P^2}(-1) \ar[d]^{\alpha'} \ar[r] &  \V \boxtimes \cO_{\P^2}(-1)  \ar[d]^\alpha  \ar[r] &  \V'' \boxtimes \cO_{\P^2}(-1)  \ar[d]^{\alpha''}   \ar[r] & 0 \\
0 \ar[r] & (\V' \oplus \V') \boxtimes  \cO_{\P^2} \ar[d]^{\beta'}   \ar[r] & (\V\oplus \V\oplus \W) \boxtimes \cO_{\P^2}  \ar[d]^\beta    \ar[r] &  (\V''\oplus \V''\oplus \W) \boxtimes \cO_{\P^2}  \ar[d]^{\beta''}    \ar[r]  & 0 \\
0 \ar[r] &  \V '\boxtimes \cO_{\P^2}(1) \ar[d]  \ar[r]& \V \boxtimes \cO_{\P^2}(1)  \ar[d]  \ar[r] & \V'' \boxtimes \cO_{\P^2}(1)  \ar[d]  \ar[r] & 0 \\
& 0 & 0  & 0 
}
$$

Since $\X$ and $\X''$ are stable, the associated long exact cohomology sequence reduces to
$$ 0 \to \mathcal H^0 (E_{\X'}) \to \mathcal H^0(E_{\X})  \to \mathcal H^0(E_{\X''} ) \to \mathcal H^1(E_{\X'}) \to 0 $$
(note that $ \mathcal H^{-1} (E_{\X''})=0$ as $\alpha''$ is fiberwise injective, hence injective).
We show that $ \mathcal H^0 (E_{\X'}) = \ker\beta'/\operatorname{im} \alpha'=0$.  First we note that, thinking of $\ell_\infty$ as the line $z=0$ is $\P^2$, 
we may write $\alpha'$, $\beta'$  restricted to $T\times \ell_\infty$ as
$$\alpha'_{\vert T\times \ell_\infty }  = \begin{pmatrix} x 1_{\V'} \\ y 1_{\V'} \end{pmatrix}, \qquad \beta'_{\vert T\times \ell_\infty}  = \begin{pmatrix}-y 1_{\V'}, & x 1_{\V'}\end{pmatrix}.$$
As a simple computation shows,\footnote{This computation will be made in the proof of Lemma \ref{cohomdeg} in the case of Hirzebruch surfaces.} one has $\operatorname{im} \alpha'=\ker\beta'$ on $T\times \ell_\infty$  so that  $\mathcal H^0 (E_{\X'}) $ is zero on $T\times \ell_\infty$, hence it has rank 0. Then it must be zero as it injects into $ \mathcal H^0(E_{\X})$ which is torsion-free.

Moreover one has:
\begin{itemize}\itemsep=-2pt \item $F=\mathcal H^0(E_{\X''})$ is a torsion-free sheaf on $T\times\P^2$,  with a framing $\phi$ to the trivial sheaf on
$T\times\ell_\infty$, where $\ell_\infty$ is a line in $\P^2$. Moveover, for every closed point $t\in T$, the second Chern class of $F_{\vert \{t\}\times\P^2}$
is $n$. \item $F$ and $E=\mathcal H^0(E_{\X})$ are flat over $T$ by Proposition \ref{propflat} as $\X$ and $\X''$ are stable. 
\item $\mathcal H^1(E_{\X'})$ is a rank 0 coherent sheaf on $T\times\P^2$, supported away from $T\times\ell_\infty$.
For every closed point $t\in T$, the restriction of the schematic support of $\mathcal H^1(E_{\X'})$ to the fiber over $t$ is a length $\ell$ 0-cycle in $\P^2$.
\item $\mathcal H^1(E_{\X'})$ is flat over $T$ as it is a quotient of flat sheaves. (One can also prove this directly as in Proposition  \ref{propflat}.)
\end{itemize}
Thus the triple $(E,F,\phi)$ is a flat family of framed flags on $\P^2$ parameterized by the scheme $T$.
This defines the natural transformation $\eta$. One can indeed show that for any scheme morphism $f\colon T\to S$ the diagram 
$$\xymatrix{ \mathfrak{R}^{\boldbar {\mathbf Q}\fr  s}_{\mathbf v, \Theta}(S) \ar[rr]^{\mathfrak{R}^{\boldbar {\mathbf Q}\fr  s}_{\mathbf v, \Theta}(f)}  \ar[d]_{ \eta_S}& & \mathfrak{R}^{\boldbar {\mathbf Q}, \fr s}_{\mathbf v, \Theta}(T) \ar[d]^{ \eta_T}\\
	\mathfrak{F}_{r,c,\ell}(S) \ar[rr]^{\mathfrak{F}_{r,c,\ell}(f)} &&  \mathfrak{F}_{r,c,\ell}(T)}$$
	commutes.
	
To show that $\eta$ is actually an isomorphism one constructs a natural transformation going the opposite direction which is both a right and a left inverse to $\eta$. This is accomplished by tracing back the steps that led to the definition of $\eta$. Thus, given a family of framed flags on $\P^2$ with the required numerical invariants, one defines families of representations $\X$, $\X''$ of the standard ADHM quiver, with a surjection $\Psi \colon \X \to \X''$. Then one defines $\X'$ as the kernel of $\Psi$; the families $\X'$ and $\X$ now combine to yield a family of representations of the enhanced ADHM quiver $\boldbar {\mathbf Q}$. This concludes the proof of Theorem \ref{JvFthm}.

As we recalled in Section \ref{gen}, the functor $\mathfrak{F}_{r,c,\ell}$  is representable, so that there is a fine moduli scheme $\mathcal{F}_{r,c,\ell}$ for framed flags on $\P^2$ with numerical invariants $r$, $n$, $\ell$. So we have:

\begin{corol} \label{cormod} The moduli scheme  $\mathcal{M}^{\boldbar {\mathbf Q}\fr  s}_{\mathbf v, \Theta}$ representing the functor  $\mathfrak{R}^{\boldbar {\mathbf Q}\fr  s}_{\mathbf v, \Theta}$ is isomorphic to the moduli scheme  $\mathcal{F}_{r,c,\ell}$.
\end{corol}
 
 \section{Framed sheaves on Hirzebruch surfaces} \label{framHirz}
 
 The $n$-th Hirzbruch surface will be denoted $\Sigma_n$. We shall denote by $\mathfrak h$ and $\mathfrak e$ the cohomology classes   of the sections of the fibration $\Sigma_n\to\P^1$ that square to $n$ and $-n$, respectively, and by $\mathfrak f$ the class of the fiber. We shall use $(\mathfrak h,\mathfrak f)$ as a basis for $\operatorname{Pic}(\Sigma_n)$ over $\Z$. 
 We denote by $\ell_\infty$ the image of a section having class $\mathfrak h$ and call it ``line at infinity.''  
Note that $\Sigma_n\setminus\ell_\infty$ is isomorphic to the total space $\Xi_n$ of the line bundle $\cO_{\P^1}(-n)$.
For $n \ge 1$, we shall consider torsion-free sheaves on $\Sigma_n$ that are framed on 
 $\ell_\infty$ to the trivial sheaf. 
  Due to the framing, the first Chern class of such a sheaf $\cE$ has necessarily the form $c_1(\cE) = a\mathfrak e$ with $a\in \Z$. 
 
As we already recalled in Section \ref{gen}, by the results in \cite{BruMar}
 there exists a fine moduli scheme $\mathcal M^n(r,a,c)$ parameterizing isomorphism classes of framed sheaves $\cE$ on $\Sigma_n$, with 
 $$\rk \cE =r, \qquad c_1(\cE)=a\mathfrak e, \qquad c_2(\cE) = c.$$ This moduli space was explicitly constructed in \cite{bbr} using monadic techniques; there it was shown that the moduli scheme is nonempty if and only if the  inequality $c + \frac12na(a-1) \ge 0$ holds, and when nonempty,  it  is a (smooth, quasi-projective, irreducible) variety. Moreover in \cite{bbr} a universal framed sheaf was constructed as the cohomology of a universal monad.
 
 When $r=1$, by twisting by a line bundle we can always set $a=0$. The moduli scheme $\mathcal M^n(1,0,c)$ is isomorphic to $\operatorname{Hilb}^c(\Xi_n)$, the Hilbert scheme parametrizing length $c$ 0-cycles on $\Xi_n$. Exploiting the monadic description of 
 $\mathcal M^n(1,0,c)=\operatorname{Hilb}^c(\Xi_n)$, it was shown in \cite{bblr,bblr-irr} that  $\operatorname{Hilb}^c(\Xi_n)$ is isomorphic to a moduli space of representations of a suitable quiver with relations. Let us recall some details of this correspondence. We start by drawing the relevant quivers; we shall denote them ${\mathbf Q}^n$, where   $n$ refers to
 $\Sigma_n$. The case $n=1$ must be treated separately.\footnote{A.~King in his thesis \cite{Ki} constructed the moduli space of framed vector bundles on $\Sigma_1$ as a K\"ahler quotient of a set of linear data satisfying a nondegeneracy (stability) condition. His thesis does not contain quivers but his linear data and relations correspond to the quiver $\mathbf Q^1$ of the present paper.} The quiver ${\mathbf Q}^n$ will be called the {\em $n$-th Hirzebruch quiver}.

 
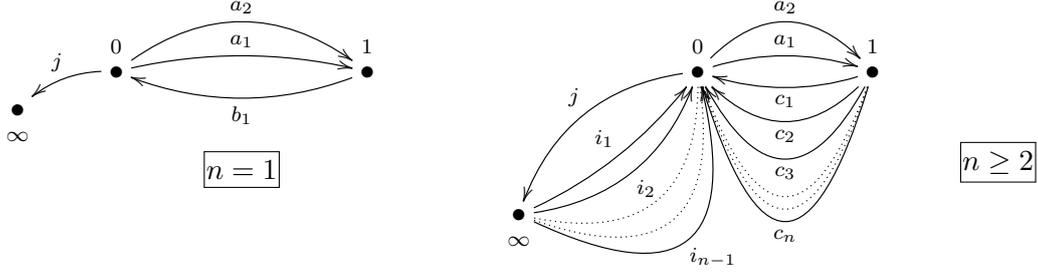
\begin{equation}
 \xymatrix@R-2.3em{
&\mbox{\scriptsize$0$}&&\mbox{\scriptsize$1$}\\
&\bullet\ar@/_1ex/[ldd]_{j}\ar@/^/[rr]^{a_{1}}\ar@/^4ex/[rr]^{a_{2}}&&\bullet\ar@/^10pt/[ll]^{b_{1}} \\ \\
\bullet&&&\\
\mbox{\scriptsize$\infty$}&&&\\ &&\mbox{\framebox[1cm]{\begin{minipage}{1cm}\centering $n=1$\end{minipage}}}
}
\xymatrix@R-2.5em{
}\qquad\qquad
  \xymatrix@R-2.3em{
&&\mbox{\scriptsize$0$}&&\mbox{\scriptsize$1$}\\
&&\bullet\ar@/_3ex/[llddddddd]_{j}\ar@/^/[rr]^{a_{1}}\ar@/^4ex/[rr]^{a_{2}}&&\bullet\ar@/^/[ll]^{c_{1}}
\ar@/^4ex/[ll]^{c_2}\ar@/^7ex/[ll]^{c_3}\ar@{..}@/^10ex/[ll]\ar@{..}@/^11ex/[ll]
\ar@/^12ex/[ll]^{c_n}& \\ \\ \\ \\ \\&&&&&\mbox{\framebox[1cm]{\begin{minipage}{1cm}\centering $n\geq2$\end{minipage}}} \\ \\
\bullet\ar@/_/[rruuuuuuu]^{i_1}\ar@/_3ex/[rruuuuuuu]_{i_2}\ar@{..}@/_6ex/[rruuuuuuu]\ar@{..}@/_8ex/[rruuuuuuu]\ar@/_10ex/[rruuuuuuu]_{i_{n-1}}&&&&&\\
\mbox{\scriptsize$\infty$}&&&&
}
\xymatrix@R-2.5em{
} \label{Hirzquiv}
\end{equation}

Denote by $\mathcal P_n$ the path algebra of the quiver ${\mathbf Q}^n$, and let $J_n$ be the ideal in $\mathcal P_n$ generated by the relations
$$ a_1c_1a_2=a_2c_1a_1 \qquad\text{when}\ n=1; $$
\begin{equation}\label{relQn} a_1c_q =a_2c_{q+1}, \qquad c_qa_1-c_{q+1}a_2 =i_qj,\ \ q=1,\dots,n-1 \qquad\text{when}\ n>1.\end{equation}
 Let $\mathcal B_n=\mathcal P_n/J_n$, fix the dimension vector $\mathbf v=(c,c,1)$ (vertexes ordered as $(0,1,\infty)$), and consider in the space of stability parameters the open cone
 \begin{equation}\label{eq:gamma_c}
  \Gamma_c=\left\{\
\Theta=(\theta_0,\theta_1)\in\mathbb{R}^2\mid\theta_0>0\,,-\theta_0<\theta_1<-\frac{c-1}{c}\theta_0
  \right\}\,.\footnote{$\theta_0$ and $\theta_1$ are the components corresponding to the vertexes $0$ and $1$, the third is fixed by the normalization
 $\sum_iv_i\cdot\theta_i=0$.}
 \end{equation}
 The cone $\Gamma_c$ is in fact a chamber (see Remark 3.3 in \cite{bblr-irr}), so that it makes sense to consider the space of framed representations of the algebra $\mathcal B_n$ with dimension vector $\mathbf v$, stable with respect to any stability parameter $\Theta$ inside $\Gamma_c$. We denote this space by $\operatorname{Rep} (\mathcal B_n,\mathbf v)_{\Gamma_c}^{\fr s}$. 
 The main theorem in \cite{bblr} and Theorem 3.8 in \cite{bblr-irr} yield
\begin{thm}  For every $n\ge 1$ and $c\ge 1$ the Hilbert scheme $\operatorname{Hilb}^c(\Xi_n)$ is isomorphic to the   GIT quotient
$$ \operatorname{Rep}(\mathcal B_n,\mathbf v)_{\Gamma_c}^{\fr s}\, /\!\!/\,\GL_c(\C) \times \GL_c(\C).$$
\label{mainbblr} \end{thm}
 We shall see later on that this scheme represents the functor of families  of framed stable representations of the quiver ${\mathbf Q}^n$.
 To that end we introduce:
 \begin{itemize} \itemsep=-2pt 
 \item the category $\A _n$ of families of  representations of the quiver ${\mathbf Q}^n$ with the relations \eqref{relQn}. For $n\ge 2$, an object of $\A _n$ is a collection 
 $$(T,\V_0,\V_1,\W,A_1,A_2,C_1,\dots,C_n,I_1,\dots,I_{n-1},J)$$
 where
 \begin{itemize} \itemsep=-2pt  \item $T$ is a scheme;  \item $\W,\V_0,\V_1$ are vector bundles on $T$;
 \item $A_1,A_2\in \Hom(\V_0,\V_1)$, $C_1,\dots,C_n\in\Hom(\V_1,\V_0)$, $I_1,\dots,I_{n-1}\in\Hom(\W,\V_0)$, \\
 $J\in\Hom(\V_0,\W)$  satisfying the conditions
$$A_1C_q=A_2C_{q+1}, \quad C_qA_1-C_{q+1}A_2 = I_q J, \quad q=1,\dots,n-1.$$
 \end{itemize}
 For $n=1$ the objects are collections $(T,\V_0,\V_1,\W,A_1,A_2,C_1,J)$ with $A_1C_1A_2=A_2C_1A_1$. 
 \item For a fixed $\mathbf v = (c_0,c_1,r)$, $\A _n(\mathbf v)$ is the full subcategory of $\A _n$ of families of representations of
 ${\mathbf Q}^n$ with dimension vector $\mathbf v$, i.e., $\rk \V_0=c_0$, $\rk \V_1 = c_1$, and  $\rk \W=r$.
 \item For a fixed stability parameter $\Theta$, $\A _n(\mathbf v)_\Theta^s$ is the full subcategory of $\A _n(\mathbf v)$ whose objects are framed representations that are stable with respect to $\Theta$.
 \item The category $\Kom_n$ of  families of complexes of cohererent sheaves on the variety $\Sigma_n$.
 \item Its full subcategory $\Kom_n^{\mbox{\tiny flat}}$ whose objects are families of complexes of cohererent sheaves on  $\Sigma_n$
 whose cohomology sheaves are flat on the base scheme.
 \end{itemize} 
 
 Morphisms in these categories are defined as in the previous Section in the case of $\P^2$.
 
 The next step would be to define a functor $\A _n\to \Kom_n$. However we are unable to do that in full generality, and we need to restrict to representations  satisfying a  kind of nondegeneracy condition, corresponding to the regularity of the pencil $\nu_1A_1 + \nu_2A_2$, where $\bnu=[\nu_1,\nu_2]\in \P^1$
 (see condition (P2) in \cite{bblr}, p.~2137).
We consider   a   full subcategory  $\A_{n,\boldsymbol \nu}$    characterized by the condition that the homomorphism
 $$ A_{\bnu} =  \nu_2A_1 + \nu_1A_2$$
is an  isomorphism. Of course this fixes the second and third components of the dimension vector to be equal.

We want to define a functor $$\K_{n,\boldsymbol \nu} \colon \A_{n,\boldsymbol \nu}\to \Kom_n$$ of categories over $ \mbox{\bf Sch}$.

We  recall that we may represent the $n$-th Hirzebruch surface $\Sigma_n$ as 
\begin{equation}
\Sigma_n = \left\{([y_1,y_2],[x_1,x_2,x_3])\in\mathbb{P}^1\times\mathbb{P}^2\;|\;x_1y_1^n=x_2y_2^n \right\}\,,
\end{equation}
and for every $\bnu=[\nu_1,\nu_2]\in\mathbb{P}^1$ we introduce the additional pair of coordinates
\begin{equation*}
[y_{1,\bnu},y_{2,\bnu}]=[\nu_1y_{1}+\nu_2y_{2},-\nu_2y_{1}+\nu_1y_{2}]\,.
\label{beta10}
\end{equation*}
The set $\left\{y_{2,\bnu}^{q}y_{1,\bnu}^{h-q}\right\}_{q=0}^{h}$ is a basis for $H^{0}\left(\mathcal{O}_{\Sigma_n}(0,h)\right)=H^{0}\left(\pi^{*}\mathcal{O}_{\mathbb{P}^1}(h)\right)$ for all $h\geq1$, where $\pi\colon\Sigma_{n}\longrightarrow\mathbb{P}^1$ is the canonical projection. Furthermore  the (unique up to homotheties) global section  $s_{\mathfrak e}$  of $\mathcal{O}_{\Sigma_n}(\mathfrak e)$ induces an injection $\mathcal{O}_{\Sigma_n}(0,n)\rightarrowtail\mathcal{O}_{\Sigma_n}(1,0)$, so that the set \[\left\{(y_{2,\bnu}^{q}y_{1,\bnu}^{n-q})s_{\mathfrak e}\right\}_{q=0}^{n}\cup\{s_{\infty}\}\] is a basis for $H^{0}\left(\mathcal{O}_{\Sigma_n}(1,0)\right)$, where $s_{\infty}$ is a section whose vanishing locus is   $\ell_\infty$.

We define the functor $\K_{n,\bnu}$ on objects.
If $X=(T,\V_0,\V_1,\W,A_1,A_2,C_1,\dots,C_n,I_1,\dots,I_{n-1},J)$ is an object in
 $\A_{n,\boldsymbol \nu}$,
 then $\K_{n,\bnu}(X)$ is the complex\footnote{The dual vector spaces and morphisms appear in the next formulas as the linear data chosen in \cite{bblr} locally reduce to the transposes of Nakajima's linear data. This also has other consequences: our stability corresponds to Nakajima's {\em co-stability,} and, contrary to the case of $\P^2$ in the first part of this paper, the functor  $\K_{n,\bnu}$ is contravariant.}
\begin{equation}\label{monadeventually} 0 \to \V_0^\ast \boxtimes\opq{0}{-1} \xrightarrow{\ \alpha_{\bnu}\ }   \V_0^\ast\boxtimes \opq{1}{-1} \oplus  (\V_0^\ast\oplus \W^\ast) \boxtimes \cO_{\Sigma_n}
 \xrightarrow{\ \beta_{\bnu}\ } \V_0^\ast \boxtimes\opq{1}{0} \to 0 
\end{equation}
 with the morphisms $\alpha_{\bnu}$, $\beta_{\bnu}$ given by
$$\alpha_{\bnu}=
\begin{pmatrix}
\id  \otimes (y_{2,\bnu}^{n}s_{\mathfrak e})+ {A_{\bnu}^\ast}{C^\ast_{\bnu}}\otimes s_\infty\\[5pt]
\id  \otimes y_{1,\bnu}+ {D_\bnu^\ast }  (A_{\bnu}^\ast )^{-1}  \otimes y_{2,\bnu}\\
- I_\bnu^\ast \otimes y_{2\bnu}
\end{pmatrix}\,,$$
$$\beta_{\bnu}=
\begin{pmatrix}
\id  \otimes y_{1,\bnu}+ {D_\bnu^\ast }  (A_{\bnu}^\ast )^{-1}  \otimes y_{2,\bnu},&
-\left(\id  \otimes (y_{2,\bnu}^{n}s_{\mathfrak e})+ {A_{\bnu}^\ast}{C^\ast_{\bnu}}\otimes s_\infty\right),&
J^\ast\otimes\,s_\infty
\end{pmatrix}, 
$$
where we have set
$$
C_{\bnu}={\sum_{q=1}^{n}} {\textstyle\binom{n-1}{q-1}}\,\nu_1^{n-q}\nu_2^{q-1}C_q, \quad   D_\bnu = \nu_1A_1-\nu_2A_2, \quad
I_\bnu=   (\nu_1^2+\nu_2^2) \sum_{q=1}^{n-1} {\textstyle\binom{n-2}{q-1}}\, \nu_1^{n-q-1} \nu_2^{q-1} I_q
$$
(for $n=1$ we understand that $C_\bnu=1$ and $I_\bnu=0$).

The action of $\K_{n,\bnu}$  on morphisms is defined as in the case of $\P^2$, see \eqref{morP2}. We omit the cumbersome but trivial details.

 We see now some properties of the functor $\K_{n,\bnu}$. Let $\mathcal H^\bullet$ denote the cohomology sheaves of a complex on $T\times\Sigma_n$.

 \begin{defin}
     Let $\X \in \A_{n,\bnu}$ be a family of framed representations. If the dimension vector of $\X$ is $(1,c,c)$ for some $c$, and
  $\X$ is stable with respect to any stability parameter in the chamber $\Gamma_c$, then we say that $\X$ is $\Gamma_c$-stable. The same wording will be used for the absolute case.
 \end{defin}

  \begin{prop} If $\X \in \A_{n,\bnu}$ is a family of framed representations, then $\mathcal H^{-1}(\K_{n,\bnu}(\X))=0$,
  and $\mathcal H^{0}(\K_{n,\bnu}(\X))$ is torsion-free. If $\X$ is $\Gamma_c$-stable, then  $\mathcal H^{1}(\K_{n,\bnu}(\X))=0$.
  \end{prop}
 \begin{proof} This follows from the  case $T = \operatorname{Spec} \C$,   which was  proved in  Section A.1 of  \cite{bblr}.
  \end{proof}

  Let $\A_{n,\bnu}(c)^{\fr s}$ be the subcategory of $\A_{n,\bnu}$ whose
 objects are families of $\Gamma_c$-stable framed representations.

\begin{prop} \label{Prop14} \begin{enumerate} \itemsep=-2pt  
\item $\K_{n,\bnu}$ maps $\A _{n,\bnu}(c)^{\fr s}$ into $\Kom_n^{\mbox{\tiny\rm  flat}}$.
\item If $\X \in \A_{n,\bnu}(c)^{\fr s}\cap\A_{n,\bnu'}(c)^{\fr s}$ then the complexes $\K_{n,\bnu}(\X)$ and 
 $\K_{n,\bnu'}(\X)$ are quasi-iso\-morphic.
 \end{enumerate}
 \label{propK1}
 
  Let $\X\to\X''$ be a surjective morphism in $\A_{n,\bnu}^\fr$ for a scheme $T$,  
  where $\X$ is $\Gamma_c$-stable and $\X''$ has dimension vector $(0,c-c',c-c')$.
  Let $\X'$ be the
 corresponding kernel. Then: 
 \begin{enumerate}\itemsep=-2pt  \setcounter{enumi}{2}
 \item $\X'\in \A_{n,\bnu}(c')^{\heartsuit s}$ is $\Gamma_{c'}$-stable.
 \item The sequence of morphisms of complexes of coherent sheaves on $T\times\Sigma_n$
  \begin{equation} \label{seqcompl}  0 \to \K_{n,\bnu} (\X'') \to \K_{n,\bnu} (\X)
 \to \K_{n,\bnu} (\X') \to 0\end{equation} is exact.
 \end{enumerate}
 \label{propK2}\end{prop}  
 \begin{proof} 
1. This goes exactly as in Proposition \ref{propflat}. 

2. This is essentially proved in \cite{bblr}, albeit in a different language.

 3. It follows from a direct computation. 
 
 4. The sequence \eqref{seqcompl} can be written as a diagram with three rows and three columns;
 the second and third column are complexes as in \eqref{monadeventually}, and the first column too, but with $\W=0$.
 The exactness of the rows is equivalent to the exactness of the sequence $ 0 \to \X' \to \X \to \X'' \to 0 $.
 \end{proof}
 
 We conclude this section by stating and briefly  discussing the correspondence between the functor of families of representations of the quivers ${\mathbf Q}^n$, and the Hilbert scheme functor for the varieties $\Xi_n$; that is, we categorize Theorem \ref{mainbblr}.
 
 \begin{thm} \label{isofunctHirz} Let $\mathfrak R^{n\fr  s}_{c,\Theta}$ be the functor of families of $\Gamma_c$-stable framed representations of the quiver with relations ${\mathbf Q}^n$. Let $\mathfrak{Hilb}^{c}_{\Xi_n}$ be the functor     of families of  length $c$ 0-cycles  on the variety $\Xi_n$. There is a natural transformation $\eta_n\colon \mathfrak R^{n\fr  s}_{c,\Theta} \to \mathfrak{Hilb}^{c}_{\Xi_n}$ which is an isomorphism of functors. 
 \end{thm}
 \begin{proof} The natural transformation $\eta_n$ is defined by means of the functors $\K_{n,\bnu}$, also in view of part 1 of Proposition \ref{propK1}: if $\X$ is a family of $\Gamma_c$-stable representations of ${\mathbf Q}^n$, it is
 in $\A_{n,\bnu}(c)^{\fr s}$ for some $\bnu$; then $\mathcal H^0(\K_{n,\bnu}(\X))$ is a family of length $c$ 0-cycles on $\Xi_n$.
 That $\eta_n$ is an isomorphism of functors is just the categorical way of stating Theorem \ref{mainbblr}, and ultimately is the main content of \cite{bblr}.
 \end{proof}
 The version of Remark \ref{remNaka} in the present context is that the Hilbert scheme $\mathrm{Hilb}^c(\Xi_n)$ represents the functor 
 $\mathfrak R^{n\fr  s}_{c,\Theta}$.
 
 \bigskip
 \section{Nested Hilbert schemes of $\Xi_n$ as quiver varieties} \label{trippa}
 We come now to the main result of this paper, which is the analogue of Theorem \ref{JvFthm} for rank 1 framed flags on Hirzerbruch surfaces; as we previously discussed, this result establishes an isomorphism between the moduli space of stable framed representations of a suitable quiver with relations, and the Hilbert scheme of nested 0-cycles on the varietes $\Xi_n$, i.e., the varieties obtained by removing from the $n$-th Hirzebruch surfaces $\Sigma_n$ the image of a section of the fibration $\Sigma_n \to \P^1$ squaring to $n$. The quiver is an ``enhancement'' of the Hirzebruch quiver ${\mathbf Q}^n$, which we shall denote $\boldbar {\mathbf Q}^n$, and will call the {\em $n$-th enhanced Hirzebruch quiver}.
 The quiver $\boldbar {\mathbf Q}^n$  is shown in Figure \ref{figenhHirz} and the  relations in Table \ref{relenhHirz}.

  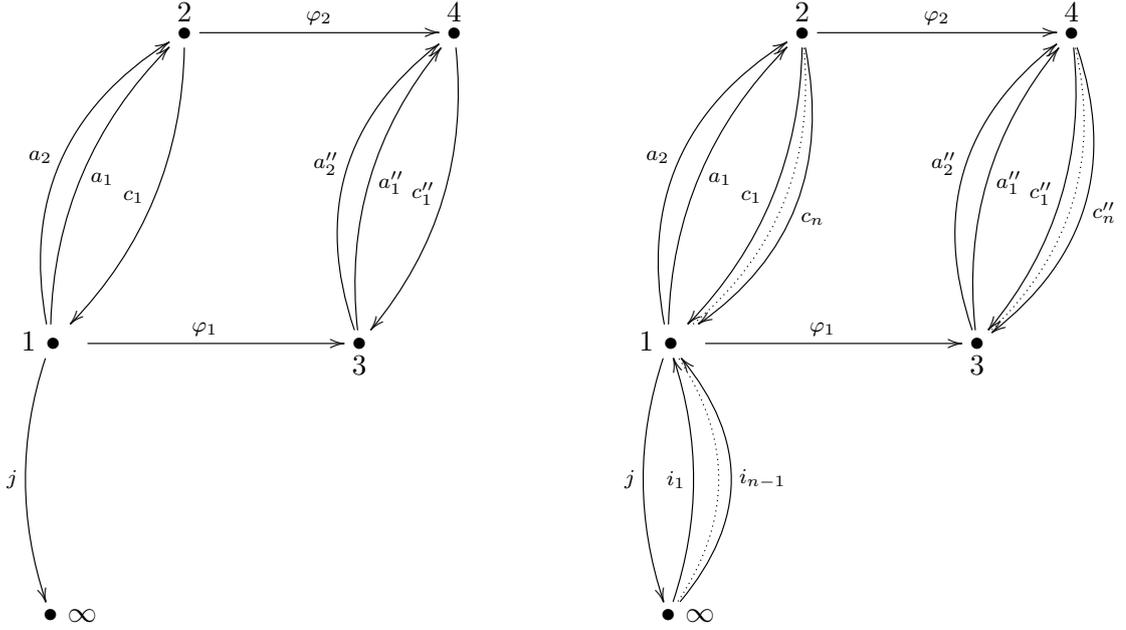
\begin{figure}  
 \begin{center}
\hskip-15mm \parbox{4cm}{ \begin{equation}
	\xymatrix@R-2.3em{
		&  \overset{\displaystyle 2}{\bullet} \ar@/^0ex/[rrr]^{\varphi_{2}}\ar@/^3ex/[lddddddddddddddddddddd]_{c_{1}}&&& \overset{\displaystyle 4}{\bullet} \ar@/^3ex/[lddddddddddddddddddddd]_{c''_{1}} 
		\\ \\ \\ \\ \\ \\ \\ \\ \\ \\ \\ \\ \\ \\ \\ \\ \\ \\ \\ \\ \\
		1 \ \bullet\ \ \ar@/^0ex/[rrr]^{\varphi_{1}}\ar@/^5ex/[ruuuuuuuuuuuuuuuuuuuuu]^{a_{2}}\ar@/^3ex/[ruuuuuuuuuuuuuuuuuuuuu]_{a_{1}}\ar@/_2ex/[ddddddddddddddddddd]_{j}&&&\underset{\displaystyle 3}{\bullet}  \ar@/^5ex/[ruuuuuuuuuuuuuuuuuuuuu]^{a''_{2}}\ar@/^3ex/[ruuuuuuuuuuuuuuuuuuuuu]_{a''_{1}}&\\ \\ \\ \\ \\ \\ \\ \\  \\ \\ \\ \\ \\ \\ \\ \\ \\ \\ \\
	\ \ 	\ \ \bullet \  \infty  &&&& 	}
\end{equation}}
\hskip4cm \parbox{4cm}{ \begin{equation}
\xymatrix@R-2.3em{
& \overset{\displaystyle 2}{\bullet}  \ar@/^0ex/[rrr]^{\varphi_{2}}\ar@/^5ex/[lddddddddddddddddddddd]^{c_{n}}\ar@{..}@/^4ex/[lddddddddddddddddddddd]\ar@/^3ex/[lddddddddddddddddddddd]_{c_{1}}&&&\overset{\displaystyle 4}{\bullet} \ar@/^5ex/[lddddddddddddddddddddd]^{c''_{n}}\ar@{..}@/^4ex/[lddddddddddddddddddddd]\ar@/^3ex/[lddddddddddddddddddddd]_{c''_{1}}
\\ \\ \\ \\ \\ \\ \\ \\ \\ \\ \\ \\ \\ \\ \\ \\ \\ \\ \\ \\ \\
	1 \ \bullet\ \  \ar@/^0ex/[rrr]^{\varphi_{1}}\ar@/^5ex/[ruuuuuuuuuuuuuuuuuuuuu]^{a_{2}}\ar@/^3ex/[ruuuuuuuuuuuuuuuuuuuuu]_{a_{1}}\ar@/_2ex/[ddddddddddddddddddd]_{j}&&&\underset{\displaystyle 3}{\bullet} \ar@/^5ex/[ruuuuuuuuuuuuuuuuuuuuu]^{a''_{2}}\ar@/^3ex/[ruuuuuuuuuuuuuuuuuuuuu]_{a''_{1}}&\\ \\ \\ \\ \\ \\ \\ \\ \\ \\ \\ \\ \\ \\ \\ \\ \\ \\ \\
\ \ 	\ \ \bullet \  \infty  \ar@/_2ex/[uuuuuuuuuuuuuuuuuuu]^{i_{1}}\ar@{..}@/_4ex/[uuuuuuuuuuuuuuuuuuu]\ar@/_5ex/[uuuuuuuuuuuuuuuuuuu]_{i_{n-1}}&&&&  }
\end{equation}}
\caption{\small \label{figenhHirz} The enhanced Hirzebruch quiver for $n=1$ (left) and $n\ge 2$ (right).}
\end{center}
 \end{figure}

\begin{table}  \centerline{\parbox{.9\textwidth}{\hrulefill}} \begin{center}  \begin{eqnarray} 
(n=1) &\quad&  a_1c_1a_2=a_2c_1a_1; \quad a_1''c_1''a_2''=a_2''c_1''a_1''; \quad \varphi_1c_1=c_1''\varphi_2; 
\\  &\quad&   \varphi_2a_1=a_1''\varphi_1;\quad \varphi_2a_2=a_2''\varphi_1 \\[3pt]
  (n\geq 2)  &\quad& 	a_1c_q =a_2c_{q+1}; \quad  c_qa_1+i_qj=c_{q+1}a_2 ; \quad a_1''c_q''=a_2''c_{q+1}; \quad 
  c_q''a_1''=c_{q+1}''a_2'';  \\
  &\quad&  \varphi_1i_q =0; \quad 
\varphi_2a_p =a_p ''\varphi_1; \quad  	\varphi_1c_t =c_t''\varphi_2   \\
 &\quad&  \text{with}\  q=1,\dots,n-1;  \quad  p=1,2; \quad  t = 1,\dots, n.
\end{eqnarray}  \end{center}
\caption{\small \label{relenhHirz} Relations for the enhanced Hirzebruch quiver $\boldbar {\mathbf Q}^n$.}
\centerline{\parbox{.9\textwidth}{\hrulefill}}
\end{table}

The result is expressed by the following theorem. As anticipated, this is a version of Theorem \ref{JvFthm} with the projective plane
replaced by a Hirzebruch surface, and restricted to the rank one case,  and also a generalization of Theorem
\ref{isofunctHirz} to the nested case.

\begin{thm} \label{mainresult} Let $\mathcal M^{n\heartsuit s}_{\mathbf v,\Theta}$ be 
			the moduli space of   framed representations of the quiver $\boldbar {\mathbf Q}^n$ with dimension vector $\mathbf (c, c, c-c', c-c',1)$,\footnote{The vertexes are taken in the order 1, 2, 3, 4, $\infty$.} stable with  respect to a stability parameter $\Theta = (\theta_1,\theta_2,\theta_3,\theta_4)$ such that
\begin{equation} \label{condTheta} (\theta_1,\theta_2)\in\Gamma_c\text{ (cf. eq.~\eqref{eq:gamma_c})},\quad \theta_3,\ \theta_4<0,  \quad 
\theta_1+\theta_2+(\theta_3+\theta_4)(c-c')>0. \end{equation}
		  $\mathcal M^{n\heartsuit s}_{\mathbf v,\Theta}$  is isomorphic to the nested Hilbert scheme $\Hilb^{c',c}(\Xi_n)$.
		\end{thm} 
 We shall prove this first at the set-theoretic level, and then, after categorizing the construction, we shall prove the scheme-theoretic isomorphism. Notice that the conditions in eq.~\eqref{condTheta} identify an open cone in the space of stability parameters (which in this case is $4$-dimensional; the last entry is fixed by the usual normalization). We do not know if this cone coincides with a chamber, but Lemma \ref{carac} ensures at least that it is part of one (i.e., all stability parameters satisfying conditions \eqref{condTheta} are generic and define the same notion of semistability).
 
The rest of this  Section is divided in three parts. In the first  we characterize the stability we are interested in.      In the second part we prove a set-theoretic version of Theorem \ref{mainresult}.
In the third part we  categorize these constructions and   conclude the proof of Theorem \ref{mainresult}.

		\subsection{Stability for the quiver $\boldbar {\mathbf Q}^n$}
		We can express the stability conditions for the quiver $\boldbar {\mathbf Q}^n$ as follows. Consider a stability parameter $\Theta =(\theta_1, \theta_2, \theta_3, \theta_4)\in \mathbb{R}^4$ and fix a dimension vector $\mathbf v =(v_1, v_2, v_3, v_4, r)\in \mathbb{N}^5$ 
		(the order of the vertices is 1, 2, 3, 4, $\infty$). 
		 Let \begin{equation} \label{reprenhHirz} X=(V_1, V_2, V_3, V_4, W, A_1, A_2, C_1, \ldots, C_n; A_1', A_2', C_1',\ldots, C_n',\ell, h_1, \ldots, h_{n-1}, F_1, F_2),\end{equation}
		where: 
		
		\noindent $\bullet $
		 $V_1, V_2, V_3, V_4$ and $W$ are $\mathbb{C}$-vector spaces of the dimensions given by $\mathbf v$; 
		 
			\noindent $\bullet$   $A_1, A_2\in \Hom(V_1, V_2)$; \ \ \ $C_1,\ldots, C_n\in \Hom(V_2, V_1)$;  \ \ \   $A_1', A_2'\in \Hom(V_3, V_4)$;  \\$C_1',\ldots, C_n'\in \Hom(V_4, V_3)$;  \ \ \  $\ell\in \Hom(V_1, W)$;  \ \ \  $h_1, \ldots, h_{n-1}\in \Hom(W, V_1)$;  \\ $F_1\in \Hom(V_1, V_3)$;   \ \ \  $F_2\in \Hom(V_2, V_4)$
			
			\noindent (for $n=1$ it is understood there are no maps $h_t$);
			
			\noindent $\bullet$   the maps   satisfy the relations
			$$ 
			(n=1)  \ \ 	A_1C_1A_2=A_2C_1A_1; \quad A_1''C_1''A_2''=A_2''C_1''A_1'' ; \quad  F_1C_1=C_1''F_2; 
			\quad F_2A_1=A_1''F_1;\quad F_2A_2=A_2''F_1 
			$$

			 \begin{eqnarray} (n=2) &\quad& 
				A_1C_q=A_2C_{q+1},; \quad  C_qA_1+h_q\ell=C_{q+1}A_2; \quad 
				A_1'C_q'=A_2'C_{q+1}';  \\
				&\quad&  C_q'A_1'=C_{q+1}'A_2'; \quad 
				F_2A_p=A_p'F_1; \quad F_1C_t=C_t'F_2,; \quad 
				F_1h_q=0 \\ &\quad& \text{with} \ q=1,\dots,n-1; \ \ p=1,2; \ \ t=1,\dots,n.
			\end{eqnarray}
	We recall from \cite{Ginz} (Prop.~5.1.5) the following characterization of semistability.
	 \begin{lemma} The representation $X$ is $\Theta $-semistable if the following conditions hold:
	 \begin{enumerate}
			\item $\Theta \cdot\dim(S):=\theta_1s_1+\theta_2s_2+\theta_3s_1'+\theta_4s_2'\leq 0$ 
			for all subrepresentations $S=(S_1, S_2, S_1', S_2')$ such that $S_1\subseteq \ker(\ell)$; 		
			\item $\Theta \cdot\dim(S)\leq \Theta \cdot\dim(X)$ for all subrepresentations $S=(S_1, S_2, S_1', S_2')$ such that $S_1\supseteq \mathrm{Im}(h_i)$ for $i\in \{1, \ldots, n-1\}$.
			\end{enumerate}
		$X$ is $\Theta $-stable if the inequalities are strict for $0\neq S\subsetneq X.$
		\end{lemma}
		Now we prove two Lemmas that allow us to characterize the stable representations of the quiver $\boldbar {\mathbf Q}^n$.
						\begin{lemma} \label{carac}
			Consider $\Theta =(\theta_1, \theta_2, \theta_3, \theta_4)\in \mathbb{R}^4.$ Suppose that
			\begin{itemize}
                \item $(\theta_1,\theta_2)\in\Gamma_c \text{ (cf. eq.~\eqref{eq:gamma_c})};$
				\item $\theta_3,\ \theta_4<0;$
				\item $\theta_1+\theta_2+(\theta_3+\theta_4)(c-c')>0.$
			\end{itemize}
			Let $X=(V_1, V_2, V_3, V_4, W, A_1, A_2, C_1, \ldots, C_n, A_1', A_2', C_1', \ldots, C_n', \ell, h_1, \ldots, h_{n-1}, F_1, F_2)$ be a representation  of $\boldbar {\mathbf Q}^n$ with dimension vector $(c, c, c-c', c-c', 1)$. The following statements are equivalent:
			\begin{enumerate}
				\item $X$ is $\Theta $-stable;
				\item $X$ is $\Theta $-semistable;
				\item $X$ satisfies the following conditions:
				\begin{itemize}
					\item [(C1)] $F_1$ and $F_2$ are surjective;
					\item [(C2)] $\bar X:=(V_1, V_2, W, A_1, A_2, C_1, \ldots, C_n, \ell, h_1, \ldots, h_{n-1})$ is a $\Gamma_c$-stable representation of the quiver ${\mathbf Q}^n.$ 
				\end{itemize} 
			\end{enumerate}
		\end{lemma}
Before proving the lemma, we need to recall the following characterization of $\Gamma_c$-stability (Lemma 4.7 of \cite{bblr}):
\begin{lemma}\label{caractbblr}
    A framed representation $$(V_0,V_1,W,A_1,A_2,C_1,\dots,C_n,I_1,\dots,I_{n-1},J)$$ of the quiver $\mathbf{Q}^n$ is $\Gamma_c$-stable if and only if:
    \begin{enumerate} \itemsep=-2pt
        \item For all subrepresentations $(S_0,S_1)$ such that $S_0\subseteq\ker J$, $\dim S_0<\dim S_1$, unless $S_0=S_1=0$.
        \item For all subrepresentations $(S_0,S_1)$ such that $S_0\supseteq\mathrm{Im}\, I_1,\dots,\mathrm{Im}\, I_{n-1}$, $\dim S_0\leq \dim S_1$.
    \end{enumerate}
\end{lemma}

		\begin{proof}[Proof of Lemma \ref{carac}]
			If $X$ is $\Theta $-stable, then it is obviously $\Theta $-semistable. We assume that $X$ is $\Theta $-semistable and prove (C1) and (C2). Note that $\tilde X :=(V_1,V_2, \mathrm{Im}(F_1), \mathrm{Im}(F_2))$ is a subrepresentation of $X$ such that $V_1=S_1\supseteq \mathrm{Im}(h_i)$ for all $i\in \{1, \ldots, n-1\}.$ By the $\Theta $-semistability of $X$ we have 
			$$\theta_3\dim(\mathrm{Im}(F_1))+\theta_4\dim(\mathrm{Im}(F_2))\leq\theta_3(c-c')+\theta_4(c-c').$$
			Since $\theta_3$ and $\theta_4$ are negative, we also have
			$$\theta_3\dim(\mathrm{Im}(F_1))+\theta_4\dim(\mathrm{Im}(F_2))\geq\theta_3(c-c')+\theta_4(c-c').$$
			By combining these inequalities we have 
			$$\theta_3\dim(\mathrm{Im}(F_1))+\theta_4\dim(\mathrm{Im}(F_2))=\theta_3(c-c')+\theta_4(c-c').$$
			Therefore
			$$\dim(\mathrm{Im}(F_1))=\dim(\mathrm{Im}(F_2))=c-c',$$
			i.e.,  $F_1$ and $F_2$ are surjective.
			
			Now we prove (C2) by using the characterization of Lemma \ref{caractbblr}.  Consider a nonzero subrepresentation $S=(S_1, S_2)$ of $\bar X$ such that $S_1\subseteq \ker(\ell)$. Then  
			$$\tilde X :=(S_1, S_2,V_3,V_4)$$
			is a subrepresentation of $X$ such that $S_1\subseteq \ker(\ell).$ By the $\Theta $-semistability of $X$ 
		$$\theta_1s_1+\theta_2s_2+(\theta_3+\theta_4)(c-c')\leq 0$$ where $s_1=\dim S_1$ and $s_2=\dim S_2$.
			Suppose $s_1\geq s_2$, and $s_1\neq 0$. We have   $\theta_2s_2\geq\theta_2s_1$, since $\theta_2<0.$
			Thus,
			\begin{multline}
				\theta_1s_1+\theta_2s_2+(\theta_3+\theta_4)(c-c')  \ge   (\theta_1+\theta_2)s_1+(\theta_3+\theta_4)(c-c')
				\ge \theta_1+\theta_2+(\theta_3+\theta_4)(c-c')
				>  0,
			\end{multline} 
			a contradiction. So $s_1<s_2$, as wanted. 
			
			On the other hand, suppose that $S=(S_1, S_2)$ is proper and $S_1\supseteq \mathrm{Im}(h_t)$ for every $t\in\{1, \ldots, n-1\}.$ As before, we can consider $\tilde X :=(S_1, S_2, V_3,V_4)$ as a subrepresentation of $X$, as we have the relations between the maps inherited from the path algebra. Using again the  $\Theta $-semistability of $X$ we obtain  
			$$\theta_1s_1+\theta_2s_2\leq (\theta_1+\theta_2)c.$$
			Since $(\theta_1, \theta_2)$ are in $\Gamma_c$ we know that $$\theta_1+\theta_2<\dfrac{\theta_1}{c}.$$ Then, we conclude that
			$$\theta_1s_1+\theta_2s_2<\theta_1 \iff \theta_1(s_1-1)+\theta_2s_2<0.$$
			If $s_1>s_2$ we obtain
			$$(\theta_1+\theta_2)s_2\leq \theta_1(s_1-1)+\theta_2s_2<0.$$
			However, this cannot happen, since $\theta_1+\theta_2>0$ and $s_2\geq 0.$
			So $s_1\leq s_2$ as wanted. 
			
			Finally we prove that (3) implies (1). Consider a nonzero representation $\tilde X =(S_1, S_2, S_1', S_2')$ of $X,$ such that $S_1\subseteq \ker(\ell).$ If $(S_1, S_2)$ is a nonzero subrepresentation of $\bar X$ we obtain   $s_1<s_2$. As we are under the hypothesis that $(\theta_1, \theta_2)$  is in $\Gamma_c$  we see that 
			$$\dfrac{c-1}{c}<-\dfrac{\theta_2}{\theta_1}<1.$$
			Since $s_1<s_2\leq c$,
			$$\dfrac{s_1}{s_2}\leq \dfrac{c-1}{c}<-\dfrac{\theta_2}{\theta_1}.$$
			Consequently, we must have
			\begin{equation}\label{p1}
				\theta_1s_1+\theta_2s_2<0.
			\end{equation}
			Also, as   $\theta_3$ and $\theta_4$ are both negative 
			\begin{equation} \label{p2}
				\theta_3s_1'+\theta_4s_2'\leq 0.
			\end{equation}
			Adding these two inequalities we obtain the desired result.
			 
			If  $s_1=s_2=0$ then $s_i'>0$ for some $i\in \{1, 2\}$. Again, as  $\theta_3$ and $\theta_4$ are   negative,  one has
			$$\theta_1s_1+\theta_2s_2+\theta_3s_1'+\theta_4s_2'=\theta_3s_1'+\theta_4s_2'<0.$$

			By the $(\theta_1, \theta_2)$-stability of $\bar X$, the maps $h_1, \ldots, h_{n-1}$ are zero (see Theorem 3.8 in \cite{bblr-irr}), and we need to to prove that for any proper subrepresentation one has 
			$$\Theta \cdot(s_1, s_2, s_1', s_2')<\Theta \cdot(c, c, c-c', c-c').$$
			Using again the condition (C2) we obtain   $s_1\leq s_2$ and then $\theta_2s_2\leq \theta_2s_1$ as  $\theta_2<0$.
			Therefore, 
			$$\theta_1s_1+\theta_2s_2+\theta_3s_1'+\theta_4s_2'\leq (\theta_1+\theta_2)s_1+\theta_3s_1'+\theta_4s_2'.$$
			If $s_1=c$, we also have $s_2=c$ and then   $S=X$, as $F_1$ and $F_2$ are surjective. Thus  we can assume   $s_1\leq c-1$. So, 
			\begin{equation}\label{1p}
				\theta_1s_1+\theta_2s_2+\theta_3s_1'+\theta_4s_2'\leq (\theta_1+\theta_2)(c-1)+\theta_3s_1'+\theta_4s_2'.
			\end{equation}
			By hypothesis
			$$\theta_1+\theta_2+(\theta_3+\theta_4)(c-c')>0.$$
			So, 
			$$\theta_3s_1'+\theta_4s_2'<\theta_1+\theta_2+(\theta_3+\theta_4)(c-c')$$
			as $\theta_3s_1'+\theta_4s_2'\leq 0.$
			Adding $(\theta_1+\theta_2)(c-1)$ to both sides of this inequality  we have 
			\begin{equation}\label{2p}
				(\theta_1+\theta_2)(c-1)+\theta_3s_1'+\theta_4s_2'<(\theta_1+\theta_2)c+(\theta_3+\theta_4)(c-c').
			\end{equation}
			By combining \eqref{1p} and \eqref{2p}  we obtain the result:
			$$\theta_1s_1+\theta_2s_2+\theta_3s_1'+\theta_4s_2'<(\theta_1+\theta_2)c+(\theta_3+\theta_4)(c-c').$$ 
		\end{proof}
		\begin{lemma}\label{rmkcarac}
			Given $\Theta $ as in Lemma \ref{carac} and a $\Theta $-stable representation $X$  of 
			$\boldbar {\mathbf Q}^n$ with dimension vector $(c, c, c-c', c-c', 1)$, we can construct a $\Gamma_{c'}$-stable representation $X'$   by letting
			$$X':=(\bar V{_1}, \bar V{_2}, W, \bar {A_1}, \bar{A_2}, \bar{C_1}, \ldots, \bar{C_n}, \overline{\ell}, \bar h{_1}, \ldots, \bar h{_{n-1}}),$$
			where $\bar V{_1}=\ker(F_1)$, $\bar V{_2}=\ker(F_2)$ and the maps $\bar{A_1}, \bar{A_2}\in \Hom(\bar V{_1}, \bar V{_2})$, $\bar{C_1}, \ldots, \bar{C_n}\in \Hom(\bar V{_2}, \bar V{_1})$, $\overline{\ell}\in \Hom(\bar V{_1}, W)$ and $\bar h{_1}, \ldots, \bar h{_{n-1}}\in \Hom(W, \bar V_1)$ are defined by:
			\begin{eqnarray*}
				\bar{A_i}&=&A_i|_{\ker(F_1)},\ \ \mbox{for}\ i=1, 2; \quad
				\bar{C_i} = C_i|_{\ker(F_2)},\ \ \mbox{for}\ i=1, \ldots, n;\\[3pt]
				\bar h{_q}&=& h_q,\ \ \mbox{for}\ q=1, \ldots, n-1; \quad
								\overline{\ell} = \ell|_{\ker(F_1)}.
			\end{eqnarray*}
		\end{lemma}
		\begin{proof}
			We start by noticing that these maps are well-defined. In fact we have   $\mathrm{Im}(\bar{A_i})\subseteq \ker(F_2)$ and $\mathrm{Im}(\bar{C_i})\subseteq \ker(F_1)$, since $F_2A_i=A_i'F_1$ for $i\in\{1,2\}$ and $F_1C_i=C_i'F_2$ for $i\in\{1,\ldots, n\}$. The fact that we set $\bar{h}_q=h_q$ is justified by the relations $F_1h_q=0$. Moreover, all maps satisfy the necessary relations in   a straightforward way, for instance: take $x\in\ker(F_2)$ and observe that
			$$\bar{A_1}\bar{C_i}(x)=\bar{A_1}C_i(x)=A_1C_i(x)=A_2C_{i+1}(x)=\bar{A_2}\bar{C}_{i+1}(x).$$
			Hence, $\bar{A_1}\bar{C_i}=\bar{A_2}\bar{C}_{i+1}$. The other relations can be proved in an analogous way. For the $\Gamma_{c'}$-stability, consider a subrepresentation $S=(S_1, S_2)$ of $X'$ such that $S_1\subseteq \ker(\overline{\ell})$; one can see $S$ as a subrepresentation of $\bar X$ (notation of Lemma \ref{carac}), since we have the natural inclusions $i_1: \bar V{_1}=\ker(F_1)\to V_1$ and $i_2: \bar V{_2}=\ker(F_2)\to V_2,$ and we also know that 
			$$S_1\subseteq \ker(\overline{\ell})=\ker(\ell|_{\ker(F_1)})=\ker(\ell)\cap\ker(F_1)\subseteq \ker(\ell).$$ The $\Gamma_c$-stability of $\bar X$ implies that $s_1 < s_2$ or $s_1=s_2=0.$  By the same token, we use   $\bar h{_i}=h_i$ to show that if $S=(S_1, S_2)$ is a subrepresentation of $X'$ with $S_1\supseteq \mathrm{Im}(\bar h{_i})=\mathrm{Im}(h_{i})$, then by using    the $\Gamma_c$-stability of $\bar X$ one gets  $s_1\leq s_2$. Thus $X'$ is $\Gamma_{c'}$-stable by Lemma \ref{caractbblr}.	\end{proof}
		
		\subsection{The set-theoretic correspondence} \label{setcorr}
		Let $\operatorname{Rep}(\boldbar {\mathbf Q}^n)^{\heartsuit s}_{\mathbf v,\Theta}$ be the space of representations of the enhanced Hirzebruch quiver $\boldbar {\mathbf Q}^n$,
		with the relations of Table \ref{relenhHirz} and dimension vector $\mathbf v = (c, c, c-c', c-c',1)$, stable with  respect to a stability parameter $\Theta = (\theta_1,\theta_2,\theta_3,\theta_4)$ satisfying the conditions in equation \eqref{condTheta}.
This space is acted upon by a group $G$
$$ G = \GL(V_1)\times\GL(V_2)\times\GL(V_3)\times\GL(V_4) \simeq  \GL_c (\C)\times \GL_c (\C) \times \GL_{c-c'}  (\C) \times \GL_{c-c'}  (\C) $$
where $V_i$ is the vector space attached to the $i$-th vertex.
The action is, with reference to equation \eqref{reprenhHirz},
\begin{multline}\label{action}
	(g_1, g_2, g_3, g_4) (A_i, C_i, h_i,\ell, A_i'', C_i'', F_1, F_2)=\\ (g_2A_ig_1^{-1}, g_1C_ig_2^{-1}, g_1h_i,\ell g_1^{-1}, g_4A_i''g_3^{-1}, g_3C_i''g_4^{-1}, g_3F_1g_1^{-1}, g_4F_2g_2^{-1}).
\end{multline}
One can show that the stable representations are free points for this action (see
\cite{TesePedro} for a proof).

 		\begin{thm}\label{BLS}
		 There is a set-theoretical bijection between the moduli space of stable framed representations $\operatorname{Rep}(\boldbar {\mathbf Q}^n)^{\heartsuit s}_{\mathbf v,\Theta}/\!\!/G$ and the nested Hilbert scheme
$\Hilb^{c',c}(\Xi_n)$.
		\end{thm} 

  Before proving this result, we shall define a (contravariant) functorial correspondence $X \rightsquigarrow M(X)$ between representations of the ADHM quiver and complexes on $\Sigma_n$; this will be a local version of the functor $\K_n$.
  
  Let $V$, $W$ be vector spaces of dimension $c$ and 1, respectively, and let 
		  $(b_1, b_2, e)\in \mathrm{End}(V)^{\oplus 2}\oplus \Hom(V, W)$, with $[b_1,b_2]=0$. This may be regarded as 
		  a representation $X=(V,W,b_1,b_2,e,0)$ of the ADHM quiver. Taking inspiration from \cite[p. 2151]{bblr} we define the complex $M(X)$ on $\Sigma_n$ to be: 
		\begin{equation}
			M(X):\ \ 0\to\cO _{\Sigma_n}(0, -1) \otimes V^\ast \stackrel{\alpha}{\longrightarrow} \cO _{\Sigma_n}(1, -1)\otimes V^\ast\oplus \cO _{\Sigma_n} \otimes [V^\ast \oplus W^\ast] \stackrel{\beta}{\longrightarrow} \cO _{\Sigma_n}(1,0) \otimes V^\ast \to 0,
		\end{equation}
		where
		\begin{equation}
			\alpha=\left[ \begin{array}{c}
				\id_{V^\ast}(y_{2}^ns_{\mathfrak{e}}) +b^\ast_2s_{\infty}\\ \id_{V^\ast}y_{1}+b^\ast_1y_{2} \\ 0
			\end{array}\right]
		, \qquad  
			\beta=\left[\begin{array}{ccc}
				\id_{V^\ast}y_{1}+b^\ast_1y_{2}, & -\left( \id_{V^\ast}(y_{2}^ns_{\mathfrak{e}}) +b^\ast_2s_{\infty}\right),  & e^\ast s_{\infty} \end{array} \right]. 
		\end{equation}
		Moreover, given a morphism of ADHM representations $\varphi: X\to\tilde X $, where $\tilde X$ is a representation with the same  structure, we define the morphism between the complexes $M(\tilde X)$ and $M(X)$ to be
		\begin{equation}\label{diagmorphism}
			\xymatrix{
				0\ar[r]&\cO _{\Sigma_n}(0, -1)\otimes \tilde V^\ast  \ar[r]^-{\widetilde{\alpha}}\ar[d]^-{\varphi_\U }&\cO _{\Sigma_n}(1, -1)\otimes \tilde V^\ast\oplus\cO _{\Sigma_n}\otimes[\tilde V^\ast\oplus  W^\ast] \ar[r]^-{\widetilde{\beta}}\ar[r]^-\beta\ar[d]^{\varphi_\V }& \cO _{\Sigma_n}(1,0)\otimes \tilde V^\ast\ar[r]\ar[d]^{\varphi_\W }&0\\
				0\ar[r] &\cO _{\Sigma_n}(0, -1)\otimes V^\ast \ar[r]^-\alpha&\cO _{\Sigma_n}(1, -1)\otimes V^\ast\oplus\cO _{\Sigma_n}\otimes [V^\ast\oplus W^\ast]\ar[r]^-{\beta}& \cO _{\Sigma_n}(1,0)\otimes V^\ast \ar[r]&0
			}	
		\end{equation}
	where $\varphi_\U$  and $\varphi_\W$ are   the identity times the dual of the morphism $F\colon V\to\tilde V$ in $\varphi$. The morphism
	$\varphi_\V$ is defined as
	$$\varphi_{\V }=\left[ \begin{array}{ccc}
			\id \otimes F^\ast &0&0\\
			0&\id \otimes F^\ast &0\\
			0&0&\id 
		\end{array}\right].$$
	Note that the the entries below the diagonal are forced to be zero by the vanishing $H^0(\cO_{\Sigma_n}(-1,1))=0$.
	The check that this defines a morphism of complexes is done in detail in \cite{TesePedro}.

 In what follows, we shall consider monomorphisms in the category of framed representations of the ADHM quiver; the corresponding quotient representation will be framed to the zero vector space. The next lemma provides cohomological vanishings for the complexes on $\Sigma_n$ which arise from such ``degenerate'' representations through the correspondence $X \rightsquigarrow M(X)$ we described above (suitably adapted).

\begin{lemma}\label{cohomdeg}

Let $V$ be a vector space of dimension $c$, and $b_1,b_2$ two commuting elements in $\operatorname{End}(V)$. This may be regarded as a ``degenerate" representation $X=(V,\{0\},b_1,b_2,0,0)$ of the ADHM quiver.
			The complex $M(X)$ given by
			$$0\to\cO _{\Sigma_n}(0,-1) \otimes V^\ast \stackrel{\alpha}{\longrightarrow}( \cO _{\Sigma_n}(1,-1)  \oplus\cO _{\Sigma_n}) \otimes V^\ast \stackrel{\beta}{\longrightarrow}\cO _{\Sigma_n}(1,0) \otimes V^\ast\to 0,$$ 
			where
			\begin{equation}
				\alpha=\left[ \begin{array}{c}
					\id_{V^\ast}(y_{2}^ns_{\mathfrak{e}}) +b^\ast_{2}s_{\infty}\\ \id_{V^\ast}y_{1}+ b^\ast_{1}y_{2}
				\end{array}\right]
			, \quad 
				\beta=\left[\begin{array}{ccc}
					\id_{V^\ast}y_{1}+ b^\ast_{1}y_{2}, & -\left( \id_{V^\ast}(y_{2}^ns_{\mathfrak{e}}) +b^\ast_{2}s_{\infty}\right) \end{array} \right],
			\end{equation}
			satisfies  $$\mathcal{H}^{-1}(M(X))=\mathcal{H}^{0}(M(X))=0.$$
		\end{lemma}
		\begin{proof} First we show that $\operatorname{Im} \alpha =  \ker \beta$; of course we only have to check that $\operatorname{Im} \alpha\supset  \ker \beta$. The restrictions of $y_{1}$, $y_{2}$ to $\ell_\infty$ may be regarded as homogeneous coordinates on $\ell_\infty$;\footnote{Note that $(y_{1},y_{2})$ are sections of $\cO _{\Sigma_n}(0,1)$, which restricted to $\ell_\infty\simeq \mathbb{P}^1$ is $\cO _{\mathbb{P}^1}(1)$ as $\mathfrak h \cdot \mathfrak f =1 $.} by abuse of notation we shall denote them again by $y_1$, $y_2$.
			Moreover, notice that the section $s_{\mathfrak e}$ has no zeroes on $\ell_\infty$ (actually
			$\cO _{\Sigma_n}(1,-n)\Big|_{\ell_\infty}$ is trivial as $\mathfrak e\cdot\mathfrak h=0$). 
			Omitting to write the restriction to $ \ell_\infty$, we have
			$$\alpha =  \begin{pmatrix} \id_{V^\ast}(y_{2}^ns_{\mathfrak{e}})\\ \id_{V^\ast}y_{1}+ b^\ast_{1}y_{2}\end{pmatrix}$$
			and 
			$$\beta = \begin{pmatrix} -(\id_{V^\ast}y_{1}+b^\ast_{1}y_{2}), \  \id_{V^\ast}(y_{2}^ns_{\mathfrak{e}})\end{pmatrix}.
			$$
			So $(v_1,v_2) \in \ker \beta$ if and only if 
			\begin{equation} \label{ker}
				(y_{1} + y_{2} \,  b^\ast_{1})v_1 = y_{2}^n s_{\mathfrak e}  v_2. \end{equation}
			
			If $y_{2}\neq 0$, let $(v_1,v_2)$ satisfy \eqref{ker}, and set
			$$ v= \dfrac{v_1}{y_{2}^n s_{\mathfrak e}}.$$
			Then, taking \eqref{ker} into account, one has  $\alpha(v) = (v_1,v_2)$. 
			
			In the patch $y_1\ne 0$ the morphism $L = y_1   + y_2 \, b_{1}^\ast $ is invertible at $y_2=0$, hence it is invertible in a neighborhood of that point.
			Then setting $v  = L^{-1} v_2$ we again have  $\alpha(v) = (v_1,v_2)$ in that neighborhood. As this neighborhood and the neighborhood $y_2\ne 0$ cover $\ell_\infty$ the claim follows.
			
			Moreover, $\mathcal{H}^{-1}(M(X))=\ker(\alpha)$ and $\alpha$ is injective by \cite[Statement (i), p. 2151]{bblr}. This finalizes the proof.
		\end{proof}

  We start now the proof of Theorem \ref{BLS}. At first we shall describe some constructions that will allow us to build a nested 0-cycle  out of a representation of the enhanced Hirzebruch quiver. We start by considering an element $X$ in  $\operatorname{Rep}(\boldbar {\mathbf Q}^n)^{\heartsuit s}_{\mathbf v,\Theta}$ as in equation \eqref{reprenhHirz}. 
		By Lemma \ref{carac} the maps $F_1$ and $F_2$ are surjective. Moreover, $X$ includes the data for a representation of the Hirzebruch quiver
		${\mathbf Q}^n$, corresponding to the left side of enhanced Hirzebruch quiver $\boldbar {\mathbf Q}^n$ (see Figure \ref{figenhHirz}), which turns out to be stable.
		 The maps $F_1$ and $F_2$ can be seen as a morphism of representations of the quiver ${\mathbf Q}^n$ and therefore   Lemma \ref{rmkcarac}   yields a short exact sequence in the category of representations of the quiver ${\mathbf Q}^n$:
		\begin{eqnarray*}
			0\to X_{c'}\stackrel{i}{\longrightarrow}X_c\stackrel{F}{\longrightarrow}X_{c-c'}\to 0,
		\end{eqnarray*}
		where $X_{c'}$ is $\Gamma_{c'}$-stable, $X_{c}$ is $\Gamma_c$-stable and the maps $i$ and $F$ are  		
		$$i=(i_1, i_2, \mathrm{Id}_W), \qquad  F=(F_1, F_2, 0 ).$$

As $X_{c}$ is $\Gamma_c$-stable, we know from  \cite[Prop.~4.9]{bblr} and   \cite[Thm.~3.8]{bblr-irr} that the   pencil $A_\bnu$ is regular, and therefore there exists an $m$ in the range $0,\dots,c$ such that the map 
\begin{equation}\label{A2}A_{2m}:=s_mA_1+c_mA_2\quad \text{where}\quad s_m = \sin \tfrac{m\pi}{c+1}, \quad c_m = \cos \tfrac{m\pi}{c+1}
		\end{equation} 
		is invertible. We fix such an $m$ and notice that also the map
		\begin{equation}
			A'_{2m}:=s_mA'_1+c_mA'_2=s_mA_1|_{\ker(F_1)}+c_mA_2|_{\ker(F_1)}=A_{2m}|_{\ker(F_1)}
		\end{equation}
		 is invertible. By  varying $m$ the open sets defined by the invertibility of $A_{2m}$ cover the space of stable linear data; on the geometric side, this corresponds to   an open cover $\left\{U_m^{nc}\right\}$, $m=0, \ldots, c$, of $\Hilb^{c}(\Xi_n)$ where each open   is isomorphic to $\Hilb^{c}(\C^2)$; this $\C^2$ is $\Sigma_n$ deprived of the line at infinity and the fiber over the point $[c_m,s_m]$ of $\P^1$.   

		 	Using \cite[Proposition 3.3]{bblr} we can build a short exact sequence in the category of representations of the   ADHM quiver
		\begin{equation}
			0\to Y_{c'}\to Y_c\to Y_{c-c'}\to 0.
		\end{equation}
		where $Y_{c'}$ and $Y_{c}$ are co-stable in Nakajima's sense. So we can picture a diagram 
		\begin{equation}\label{seqexataprov}
			\begin{tikzpicture} 
				\node at (-2,0) {$0$}; \draw[->,thick] (-1.6,0) to (-0.4,0); \node at (0,0) {$\bigcirc$}  edge [thick, out=95, in=155, loop] () edge [thick, out=85, in=25, loop]  () ; \node at (0,0) {{\scriptsize $c'$}};\node at (0,-2) {$\bigcirc$};\node at (-2,-2) {$0$}; \draw[->,thick] (-1.6,-2) to (-0.4,-2); \node at (0,-2) {{\scriptsize $1$}};  \node at (-1, 0.85) {{\scriptsize $b'_{1m}$}}; \node at (1, 0.85){{\scriptsize $b'_{2m}$}};  
				\draw[->,thick,bend right] (-0.2,-0.2) to (-0.2,-1.8); \node at (-0.6, -1) {{\scriptsize $\ell'$}};\node at (3,0) {{\scriptsize $c$}};
				\node at (3,0) {$\bigcirc$} edge [thick, out=95, in=155,loop] () edge [thick, out=85, in=25, loop] ();\node at (3,-2) {$\bigcirc$}; \node at (3,-2) {{\scriptsize $1$}};  \node at (2, 0.85) {{\scriptsize $b_{1m}$}}; \node at (4, 0.85){{\scriptsize $b_{2m}$}};  
				\draw[->,thick,bend right]  (2.8,-0.2) to (2.8,-1.8);
				\draw[->,thick] (0.4,0) to (2.6,0); \draw[->,thick] (0.4,-2) to (2.6,-2); \node at (1.5, -1.7) {{\scriptsize $\id$}};
				\node at (1.5,0.3) {{\scriptsize $i_1$}};  \node at (2.4,-1) {{\scriptsize $\ell$}};\node at (6,0) {{\scriptsize $s$}};
				\node at (6,0) {$\bigcirc$} edge [thick, out=95, in=155, loop] () edge [thick, out=85, in=25, loop]  ();  \node at (5, 0.85){{\scriptsize $\bar b_{1m}$}}; \node at (7, 0.85){{\scriptsize $\bar b_{2m}$}}; \draw[->,thick] (6.4,0) to (7.6,0); 
				\draw[->,thick] (3.4,0) to (5.4,0); \node at (4.4, 0.2) {{\scriptsize $F_1$}}; \node at (8,0) {$0$};
				\draw[dashed] (1.7,1.5) to (1.7,-2.5); \draw[dashed] (1.7,1.5) to (7.3,1.5);  \draw[dashed] (7.3,-2.5) to (7.3,1.5); 
				 \draw[dashed] (7.3,-2.5)  to (1.7,-2.5);
			\end{tikzpicture}
		\end{equation}
		where $s=c-c'$ and  $\bar b_{1m}$ and $\bar b_{2m}$ are the quotient maps,     which satisfy $[\bar b_{1m}, \bar b_{2m}]=0$ as  $[b_{1m},b_{2m}]=0.$
		Note that the portion of the diagram within dashed lines is the  dual  enhanced ADHM quiver.

		By the previous construction  and   by Proposition \ref{Prop14} (4) we get a short exact sequence
		$$0\to M(Y_{c-c'})\to M(Y_c)\to M(Y_{c'})\to 0 $$
		which is the following diagram with exact rows:
		\begin{center}
			\scalebox{0.7}{
				\begin{xy}
					\xymatrix{
						&0\ar[d]&0\ar[d]&0\ar[d]&\\
						0\ar[r]&\cO _{\Sigma_n}(0, -1) \otimes V_3^\ast  \ar[r]^{\varphi_\U '}\ar[d]^{\alpha''}&\cO _{\Sigma_n}(0, -1) \otimes V_1^\ast \ar[r]^{\varphi_\U }\ar[d]^{\alpha}&\cO _{\Sigma_n}(0, -1)\otimes N^\ast \ar[d]^{\alpha'}\ar[r]&0\\
						0\ar[r]&[\cO _{\Sigma_n}(1,-1) \oplus\cO _{\Sigma_n}]  \otimes V_3^\ast  \ar[d]^{\beta''}\ar[r]^{\varphi_\V '\ \ \ \ \ \ \ }&\cO _{\Sigma_n}(1,-1) \otimes V_1^\ast \oplus\cO _{\Sigma_n} \otimes [V_1^\ast \oplus W^\ast] \ar[d]^{\beta}\ar[r]^{\varphi_\V \ \ }&\cO _{\Sigma_n}(1,-1)\otimes N^\ast \oplus\cO _{\Sigma_n} \otimes [N^\ast\oplus W^\ast] \ar[d]^{\beta'}\ar[r]&0\\
						0\ar[r]&\cO _{\Sigma_n}(1,0) \otimes V_3^\ast \ar[d]\ar[r]^{\varphi_\W '}&\cO _{\Sigma_n}(1,0) \otimes V_1^\ast\ar[d]\ar[r]^{\varphi_\W }&\cO _{\Sigma_n}(1,0)\otimes N^\ast \ar[d]\ar[r]&0\\
						&0&0&0&}
				\end{xy}
			}
		\end{center}
		where $N = \ker F_1 \colon V_1 \to V_3$. 
  The morphisms $\varphi'$ are defined in analogy with the morphisms $\varphi$.
		The corresponding long exact sequence of cohomology contains the segment 
		\begin{equation} \label{longex}
			\mathcal{H}^{0}(M(Y_{c-c'}))\to \mathcal{H}^{0}(M(Y_c))\to \mathcal{H}^{0}(M(Y_{c'}))\to \mathcal{H}^{1}(M(Y_{c-c'}))\to \mathcal{H}^{1}(M(Y_c)).
		\end{equation}
		However since  ${X_{c}}$ is co-stable we have 
		$\mathcal{H}^1(M(Y_c))=0$
		and by  Lemma \ref{cohomdeg}, $\mathcal{H}^{0}(M(Y_{c-c'}))$ is also zero, so that the sequence \eqref{longex} reduces to 
		\begin{equation}\label{curtafinal}
			0\to \mathcal{H}^0(M(Y_c))\to \mathcal{H}^{0}(M(Y_{c'}))\to \mathcal{H}^{1}(M(Y_{c-c'}))\to0
		\end{equation}
		The sheaves $E=\mathcal{H}^0(M(Y_c))$ and $F=\mathcal{H}^0(M(Y_{c'}))$ are rank $1$ framed torsion-free sheaves with Chern character $(1, 0, -c)$ and $(1, 0, -c')$, respectively. Moreover, 
		\begin{equation}
			\mathrm{coker}(\beta'')=\mathcal{H}^{1}(M(Y_{c-c'}))\simeq \dfrac{\mathcal{H}^0(M(Y_c))}{\mathcal{H}^0(M(Y_{c'}))}
		\end{equation}
		is a rank $0$ sheaf of length $c-c'$ supported away from $\ell_\infty$, since the ranks of $\mathcal{H}^0(M(Y_c))$ and $\mathcal{H}^0(M(Y_{c'}))$   coincide and the two sheaves are framed on $\ell_\infty$.
		Thus, we get a framed flag of sheaves  $(E, F, \varphi)$ on $\Sigma_n$, which is actually a pair of nested 0-cycles. Note that due to Proposition \ref{Prop14} (2), if we choose a different value of $m$, we obtain the same nested 0-cycles.
  
		Now  we   build the correspondence in the opposite direction, starting from  $Z\in\Hilb^{c',c}(\Xi_n)$. One can write $Z=(Z^{c'}, Z^c)$, where $Z^{c'}$ and $Z^{c}$ are $0$-cycles of length $c'$ and $c$, respectively.
  $Z^c$  for some value of $m$ belongs to an open set of the cover   $\left\{U_m^{nc}\right\}$ of $\Hilb^{c}(\Xi_n)$ introduced earlier, and $Z^{c'}$ belongs to the open set $U_m^{nc'}$ of an analogous open cover of 
 $\Hilb^{c'}(\Xi_n)$ (having $c+1$ elements as well). In this way
we obtain an element $\widetilde{Z}\in \Hilb^{c',c}(\mathbb{C}^2)$. By Theorem \ref{JvFthm}, there is a stable representation of the dual enhanced ADHM quiver with dimension vector $(s=c-c', c, 1)$ providing the portion of diagram \eqref{seqexataprov} within the dashed lines.  The diagram then is completed by taking kernels (this is analogous to the operation of taking quotients in Section \ref{projplane}, in view of the fact that here we are   working with the dual quivers). Thus obtaining the left-hand side of \eqref{seqexataprov}.

		One can rewrite this as en exact sequence of representations of the dual ADHM quiver  
		\begin{equation}
			0\to L\stackrel{i_1}{\longrightarrow} M\stackrel{F_1}{\longrightarrow} N\to 0, 
		\end{equation}
		where $L$ and $M$ are stable and $\dim L =(c', 1)$,  $\dim M=(c, 1)$, $\dim N = (c-c',0)$.  One may use \cite[Equations 3.13]{bblr}, which in this case read
		\begin{eqnarray}
			A_1&=&c_m b_{1m}+s_m\id_{V}, \quad
			A_2 = -s_m b_{1m}+c_m\id_{V},\\[3pt]
			\bar A{_1}&=&c_m\bar b_{1m}+s_m\id_{V''}, \quad
			\bar A{_2}=-s_m\bar b_{1m}+c_m\id_{V''}\\[5pt]
			\left[\begin{array}{c} C_1\\ \vdots \\ \vdots \\ C_n \end{array}\right]&=& (\sigma_m^{n-1}\otimes  \id_V)\left[\begin{array}{c}
				\id_V\\b_{1m}\\ \vdots\\ b_{1m}^{n-1}\end{array}\right]b_{2m},  \quad 
			\left[\begin{array}{c}\bar C{_1}\\ \vdots \\ \vdots \\ \bar C{_n} \end{array}\right] = (\sigma_m^{n-1}\otimes \id_{V''})\left[\begin{array}{c}\id_{V''}\\ \bar b_{1m}\\ \vdots\\ \left(\bar b_{1m}\right)^{n-1}\end{array}\right]\bar b_{2m},
		\end{eqnarray}
		to obtain a representation $X$ of the quiver $\boldbar {\mathbf Q}^n$ with dimension vector $(c, c,c-c',c-c', 1)$. This representation is depicted in Figure \ref{otherenhHirz}; the horizontal arrows are copies of the morphism $F_1$ in equation \eqref{seqexataprov}. The matrix $\sigma_m^{n-1}$ is defined by the condition
		$$ (s_m z_1+c_mz_2)^p (c_mz_1-s_mz_2)^{n-1-p} = \sum_{q=0}^{n-1} (\sigma_m^{n-1})_{pq} z_1^{n-1-q} z_2^q $$
		for $(z_1,z_2)\in \C^2$, where   $s_m$, $c_m$ are the numbers defined in equation \eqref{A2}.

	\begin{figure}\begin{center}
		 \begin{equation}
		 	\xymatrix@R-2.3em{
		 		&\bullet\ar@/^0ex/[rrr]^{F_1}\ar@/^6ex/[lddddddddddddddddddddd]^{C_{n}}\ar@{..}@/^5ex/[lddddddddddddddddddddd]\ar@/^4ex/[lddddddddddddddddddddd]_{C_{1}}&&&\bullet \ar@/^6ex/[lddddddddddddddddddddd]^{\bar C{_{n}}}\ar@{..}@/^5ex/[lddddddddddddddddddddd]\ar@/^4ex/[lddddddddddddddddddddd]_{\bar C{_{1}}}
		 		\\ \\ \\ \\ \\ \\ \\ \\ \\ \\ \\ \\ \\ \\ \\ \\ \\ \\ \\ \\ \\
		 		\bullet\ar@/^0ex/[rrr]^{F_1}\ar@/^6ex/[ruuuuuuuuuuuuuuuuuuuuu]^{A_{2}}\ar@/^4ex/[ruuuuuuuuuuuuuuuuuuuuu]_{A_{1}}\ar@/_3ex/[ddddddddddddddddddd]_{\ell}&&&\bullet \ar@/^6ex/[ruuuuuuuuuuuuuuuuuuuuu]^{\bar A{_{2}}}\ar@/^4ex/[ruuuuuuuuuuuuuuuuuuuuu]_{\bar A{_{1}}}&\\ \\ \\ \\ \\ \\ \\ \\ \\ \\ \\ \\ \\ \\ \\ \\ \\ \\ \\
		 		\bullet&&&& 
		 	}
		 \end{equation}  \caption{\label{otherenhHirz} \small The representation of the enhanced Hirzebruch quiver $\boldbar{\mathbf Q}^n$ constructed in Section \ref{setcorr}.} \end{center} \end{figure}
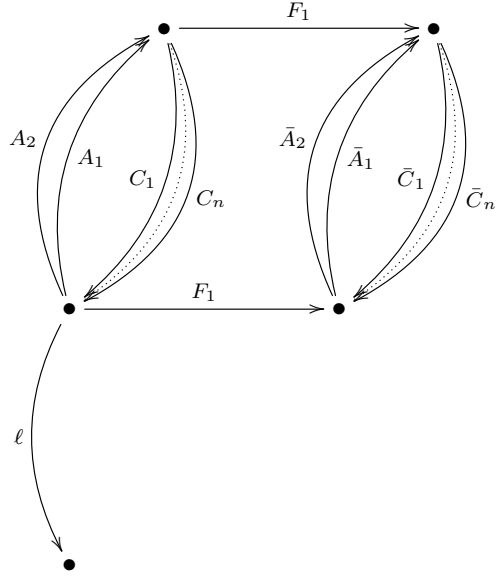

				This representation 	satisfies   the   relations in Table \ref{relenhHirz}, as one can directly check,
and is stable. In fact, the map $F_1$ is surjective and the representation of  the Hirzebruch quiver 
${\mathbf Q}^n$ given by the left portion of the quiver in Figure \ref{otherenhHirz} is stable  (see \cite{bblr} p.~2141). Then, the stability of $X$ follows from Lemma \ref{carac} and we get a point in $\operatorname{Rep}(\boldbar {\mathbf Q}^n)^{\heartsuit s}_{\mathbf v,\Theta}$. This in general depends on the choice of the open set in the cover $\left\{U_m^{nc}\right\}$ of
$\operatorname{Hilb}^c(\Xi_n)$, but its class in the quotient under the action of the group $G$ as in Theorem \ref{BLS} is actually independent of it. This construction inverts the previous one up to the action of the group $G$. The composition of the two constructions in the reverse order is the identity. This concludes the proof  of Theorem \ref{BLS}.

		\subsection{The schematic isomorphism} \label{schemiso}
		We complete the proof of Theorem \ref{mainresult} by showing that   $\mathcal M^{n\heartsuit s}_{\mathbf v,\Theta}$ and $\Hilb^{c',c}(\Xi_n)$  are isomorphic as schemes.  This will be accomplished along the lines of the proof of Theorem \ref{JvFthm} by showing that the following two functors are isomorphic:
		\begin{itemize}\itemsep=-2pt
		\item the functor $\bar{\mathfrak{R}}_{\mathbf v, \Theta }^{n\heartsuit s}$  of families of framed  representations of the enhanced Hirzebruch quiver $\boldbar {\mathbf Q}^n$ with dimension vector $\mathbf v=(c, c, c-c', c-c',1)$,  stable with respect to a stability parameter $\Theta$ as in Theorem \ref{mainresult};
		\item the functor $\mathfrak{Hilb}^{c',c}_{\Xi_n}$ of families of nested 0-cycles of length $(c',c)$ on $\Xi_n$. This is the functor $ \mbox{\bf Sch}^{\mbox{\footnotesize\rm op}}  \to   \mbox{\bf Set}$ that with every scheme $T$ associates the set  of isomorphism classes of pairs $(Z',Z)$, where $Z'$ and $Z$ are closed subschemes of $T\times \Xi_n$, flat over $T$, such that $Z'\subset Z$ and the restrictions of the projection $T\times \Xi_n \to T$ to $Z'$ and $Z$  are finite of degree $c'$ and $c$, respectively. This functor is represented by a scheme $\Hilb^{c',c}(\Xi_n)$		 by the results in \cite{Keel} (see \cite{Sernesi}, Section 4.5.1 for a detailed study of this result);  moreover, it is naturally isomorphic to the  functor   $\mathfrak{F}_{1, 0, c', c-c'}^{\Sigma_n,\ell_\infty}$ (using the notation of equation \eqref{flagfunct}) and therefore it is representable also by Theorem \ref{flagrepr}.
\end{itemize}

\begin{prop} There is a natural transformation $\eta_n \colon \bar{\mathfrak{R}}_{\mathbf v, \Theta }^{n\heartsuit s} \to \mathfrak{Hilb}^{c',c}_{\Xi_n}$ which is an isomorphism of functors.
\end{prop}
		Again, the  key for the construction of the natural transformation $\eta_n$  is to regard a representation of the enhanced Hirzebruch  quiver $\boldbar {\mathbf Q}^n$ as a morphism of representations of the Hirzebruch  quiver ${\mathbf Q}^n.$ Let 
		$$X=(T,\V_0,\V_1,\W,\V_0', \V_1',A_1,A_2,C_1,\dots,C_n,I_1,\dots,I_{n-1},J, A_1', A_2',C_1',\dots,C_n',F_1, F_2)$$
		be a family of representations of the quiver $\boldbar {\mathbf Q}^n,$ with $T$ a scheme,  $\V_0,\V_1,\W,\V_0',\V_1'$ vector bundles on $T$ of rank $c, c, 1, c-c', c-c',$ respectively, and 
		\begin{gather}
			A_1, A_2 \in \Hom(\V_0, \V_1), \quad 
			C_1, \ldots, C_n \in \Hom(\V_1, \V_0), \quad
			I_1, \ldots, I_{n-1} \in  \Hom(\W, \V_0),\\[3pt]
			J \in \Hom(\V_0, \W), \quad 
			A_1', A_2' \in  \Hom(\V_0', \V_1'),  \quad 
			C_1', \ldots, C_n' \in  \Hom(\V_1', \V_0'),\\[3pt]
			F_1 \in  \Hom(\V_0, \V_0'),  \quad 
			F_2  \in \Hom(\V_1, \V_1').
		\end{gather}
		If we assume that $X$ is stable as in Lemma \ref{carac} then $F_1$ and $F_2$ are surjective. This defines a surjective morphism of families of representations of the quiver ${\mathbf Q}^n$. Define $\V_0'':=\ker(F_1)$ and $\V''_1:=\ker(F_2)$; note that they are vector bundles on $T$ of rank $c'.$ The morphisms $A_1,A_2,C_1,\dots,C_n,J$ induce morphisms $$A_1'',A_2''\in \Hom(\V_0'', \V''_1);\ C_1'',\dots,C_n''\in \Hom(\V''_1,\V_0'');\ J''\in \Hom(\V_0'', \W)$$  and   this defines a  kernel  family of representations of the quiver ${\mathbf Q}^n.$
		
		As we have natural inclusions $i_0:\V_0''\to \V_0$ and $i_1:\V_1'\to \V_1$ and the isomorphism $\id_{\W}:\W\to\W,$ we can build a short exact sequence of families of representations of the quiver ${\mathbf Q}^n$ parameterized by $T$
		$$0\to \X''\to\X\to \X'\to 0.$$
		By Lemmas \ref{carac} and   \ref{rmkcarac}  $\X''$ and $\X$ are families of stable framed representations. According to our previous discussion, there exists a $\bnu\in\P^1$ such that we can apply the functor $\K_{n,\bnu}$ to the previous sequence, obtaining by Proposition \ref{propK2} (4)  another exact sequence
		$$ 0 \to \K_{n,\bnu} (\X') \to \K_{n,\bnu} (\X)   
		\to \K_{n,\bnu} (\X'') \to 0 .$$
By taking cohomology 		 
we obtain an exact sequence of sheaves
\begin{equation}\label{longacateg}
			\mathcal{H}^{0}(\K_{n,\bnu} (\X'))\to \mathcal{H}^{0}(\K_{n,\bnu} (\X))\to \mathcal{H}^{0}(\K_{n,\bnu} (\X''))\to \mathcal{H}^{1}(\K_{n,\bnu} (\X'))\to \mathcal{H}^{1}(\K_{n,\bnu} (\X)) ; 
		\end{equation}
		by Proposition \ref{propK2} (2) the sheaves in this sequence do not depend, up to isomorphism, on the choice of $\bnu$.
		
		On the other hand, $\mathcal{H}^{0}(\K_{n,\bnu} (\X'))=0,$ by Lemma \ref{cohomdeg} and $\mathcal{H}^{1}(\K_{n,\bnu} (\X))=0$ as  $\X$ is stable. Thus  \eqref{longacateg} reduces to 
		\begin{equation}\label{curtacateg}
			0\to \mathcal{H}^{0}(\K_{n,\bnu} (\X))\to \mathcal{H}^{0}(\K_{n,\bnu} (\X''))\to \mathcal{H}^{1}(\K_{n,\bnu} (\X'))\to 0.
		\end{equation}
		Additionally, one has: 
		\begin{itemize}
			\item $F:=\mathcal{H}^{0}(\K_{n,\bnu} (\X''))$ is a torsion-free coherent sheaf on $T\times \Sigma_n$, with a framing  $\varphi$ to the trivial sheaf on $T\times \ell_\infty.$ Moreover, the second Chern class of $F_{\vert\{t\}\times \Sigma_n}$ is $c'$ for every closed point $t\in T.$ 
			\item $F$ and $E:=\mathcal{H}^{0}(\K_{n,\bnu} (\X))$ are flat over $T$ by Proposition \ref{propK2}, since $\X$ and $\X'$ are stable.
			\item $\mathcal{H}^{1}(\K_{n,m} (\X'))$ is a rank $0$ coherent sheaf on $T\times\Sigma_n$, supported away from $T\times\ell_\infty.$ For every closed point $t\in T$, the restriction of the schematic support of $\mathcal{H}^{1}(\K_{n,\bnu} (\X'))$ to the fiber over $t$ is a $0$-cycle on $\Sigma_n$ of length $c-c'.$
			\item $\mathcal{H}^{1}(\K_{n,\bnu} (\X'))$ is flat over $T,$ as it is a quotient of flat sheaves.
		\end{itemize}
		Therefore, the pair $(E, F)$ may be regarded as  a flat  family of nested 0-cycles parameterized by the scheme $T.$ This defines the natural transformation 
		 $\eta_n$. 
		To prove that $\eta_n$ is indeed a natural transformation we need to show that for any scheme morphism $f\colon S\to T$ the diagram
		$$\xymatrix{  \bar{\mathfrak{R}}_{\mathbf v, \Theta }^{n\heartsuit s}(T) \ar[rrr]^{\bar{\mathfrak{R}}_{\mathbf v, \Theta }^{n\heartsuit s}(f)}  \ar[dd]_{\eta_{n,T}} && &\bar{\mathfrak{R}}_{\mathbf v, \Theta }^{n\heartsuit s}(S) \ar[dd]^{ \eta_{n,S}}\\
			&&&\\
			\mathfrak{Hilb}^{c',c}_{\Xi_n} (T) \ar[rrr]^{\mathfrak{Hilb}^{c',c}_{\Xi_n}(f)} &&&\mathfrak{Hilb}^{c',c}_{\Xi_n}(S)}$$
		commutes. We have indeed
		\begin{equation} 
			\mathfrak{Hilb}^{c',c}_{\Xi_n}(f)\circ \eta_{n,T}   =  (\mathrm{Id}\times f)^{*}\circ \eta_{n,T}
			=\eta_{n,S}(f^{*}) 
			= \eta_{n,S}\circ \bar{\mathfrak{R}}_{\mathbf v, \Theta }^{n\heartsuit s}(f).
		\end{equation}
		The only nontrivial equality is    the middle one. This follows from the   Lemma
			
\begin{lemma} Let $\mathcal M^\bullet$ be a family of monads on $\Sigma_n$ parameterized by a scheme $T$, and let
$f\colon S\to T$ be a scheme morphism. Assume that the cohomology $\cE$ of $\mathcal M^\bullet$ is flat over $T$.
Then $(f\times\id)^\ast \mathcal M^\bullet$  is a family of monads parameterized by $S$, whose cohomology is isomorphic to $(f\times\id)^\ast \cE$. 
\end{lemma}
\begin{proof} Denote $g=f\times\id$ and let $\mathcal H^\bullet$ be the cohomology of the complex $g^\ast \mathcal M^\bullet$. One constructs a natural morphism 
$\mathcal H^\bullet \to g^\ast \cE$, which turns out to be an isomorphism as a consequence of 
Lemma 2.2 in \cite{bbr}. To construct the morphism, write the monad $\mathcal M^\bullet$ as
$$ 0 \to \mathcal U \xrightarrow{\ \alpha\ } \mathcal V \xrightarrow{\ \beta\ } \mathcal W \to 0 .$$
One notes that $g^\ast \ker \beta \simeq \ker g^\ast\beta$ as $\mathcal W$ is locally free, while there is a morphism 
$\operatorname{im}  g^\ast \alpha\to g^\ast\operatorname{im} \alpha/\mathcal T$, where
$\mathcal T = \operatorname{Tor}_1^{\cO_{T\times\Sigma_n}}(\cO_{S\times\Sigma_n},g^{-1}(\mathcal V/\operatorname{im} \alpha))$. On the other hand, $g^\ast\cE \simeq 
g^\ast\ker \beta/(g^\ast\operatorname{im} \alpha/\mathcal T).$\end{proof}
Note that the flatness condition is fulfilled as $\X$ and $\X''$ are stable.

	 To show that $\eta_n$ is actually a natural isomorphism, we must construct  a natural transformation 
		$$\tau_n:\mathfrak{Hilb}^{c',c}_{\Xi_n}\to\bar{\mathfrak{R}}_{\mathbf v, \Theta }^{n\heartsuit s}$$
		which is both a right and left inverse to $\eta_n.$ This can be done as follows (we omit the details as they should by now be quite straightforward).
		\begin{itemize}\itemsep=-2pt
\item A  family $(Z',Z)$  of nested 0-cycles   defines two families $\X'$, $\X$ of representations of the dual ADHM quiver,
with an injective morphism $\X' \to \X$.
\item  One defines  $\X''$ as the quotient  $\X/\X'$. The exact sequence
$$ 0 \to \X' \to \X \to \X'' \to 0 $$
corresponds to a diagram as in  \eqref{seqexataprov}. 
\item   Finally one uses \cite[Equations (3.13)]{bblr} to obtain a stable family of representations of the quiver $\boldbar {\mathbf Q}^n$ with the required dimension vector and relations (this is the construction at the end of Section \ref{setcorr}).
\end{itemize} 
This finalizes the proof of Theorem \ref{mainresult}.
 
 Recalling that $\Hilb^{c',c}(\Xi_n)$ is isomorphic to the naturally defined incidence subscheme in 
$\Hilb^{c'}(\Xi_n) \times \Hilb^c(\Xi_n)$, 
we may visualize the relation between the various moduli spaces we have introduced  by the diagram
$$ \xymatrix{
\displaystyle \frac{\operatorname{Rep}(\boldbar {\mathbf Q}^n)^{\heartsuit s}_{\mathbf v,\Theta}}{G} \ar[r]   \ar[r] \ar[d] 
&  \displaystyle \frac{\operatorname{Rep}(  {\mathbf Q}^n)^{\heartsuit s}_{\mathbf v_{c'},\Gamma_{c'}}}{G_{c'}} 
\times \displaystyle \frac{\operatorname{Rep}(  {\mathbf Q}^n)^{\heartsuit s}_{\mathbf v_c,\Gamma_c}}{G_c} \ar[d]
\\
\Hilb^{c',c}(\Xi_n)  \ar[r] & \Hilb^{c'}(\Xi_n) \times \Hilb^c(\Xi_n)    
}$$ 
where $\mathbf v_c = (c,c,1)$, $\mathbf v_{c'}=(c',c',1)$ and
$$G_{c'} = \GL_{c'}(\C) \times  \GL_{c'}(\C) , \qquad G_{c} = \GL_{c}(\C)\times  \GL_{c}(\C) .$$
The vertical arrows are the isomorphisms provided by Theorems \ref{mainresult} and \ref{mainbblr} respectively, while the upper horizontal arrow consists in describing a representation of $\boldbar{\mathbf Q}^n$ as a surjective morphism between two representations of $\mathbf{Q}^n$.

\def\cprime{$'$} \def\cprime{$'$} \def\cprime{$'$} \def\cprime{$'$}

\end{document}